\newif\ifMP
\MPfalse

\ifMP
	\documentclass[smallextended]{svjour3}       % onecolumn 
\smartqed
\else
	\documentclass[11pt]{article}
\fi

\usepackage{amsmath}

\ifMP
\else
	\usepackage{amsthm}
\fi

\usepackage{amsfonts}
\usepackage{amssymb}
\usepackage{authblk}
\usepackage{graphicx}
\usepackage{color}
\usepackage{rotating}
\usepackage{url}
\usepackage{algorithm}
\usepackage{algpseudocode}
\usepackage{epstopdf}
\usepackage[table,xcdraw]{xcolor}
\usepackage[normalem]{ulem}
\useunder{\uline}{\ul}{}
\usepackage[T1]{fontenc}
\usepackage[utf8]{inputenc}
\usepackage[font=small,labelfont=bf]{caption}
\usepackage{tikz,pgfplots}
\usetikzlibrary{calc,tikzmark,arrows,shadings,shadows,automata,shapes,positioning,patterns,decorations.markings,decorations.pathmorphing,snakes,plotmarks}

\ifMP
\else
	\newtheorem{definition}{Definition}
	
	\newtheorem{theorem}{Theorem}
	\newtheorem{corollary}{Corollary}
	\newtheorem{lemma}{Lemma}

	\newtheorem{example}{Example}
\fi

\DeclareMathOperator*{\argmax}{argmax}

\DeclareMathOperator*{\conv}{Conv}

\newcommand{\red}[1]{\textcolor{black}{#1}}

\newcommand{\dist}[0]{\ensuremath{\textrm{d}}}
\newcommand{\R}{\mathbb{R}}

\newcommand{\bZ}[0]{\boldsymbol Z}

\newcommand{\A}[0]{\mathcal{A}}

\newcommand{\D}[0]{\boldsymbol D}

\newcommand{\sign}[0]{\ensuremath{\textrm{sign}}}
\newcommand{\Var}[0]{\ensuremath{\textrm{Var}}}

\newcommand{\E}[0]{\ensuremath{\mathbb{E}}}
\renewcommand{\P}{P}
\newcommand{\Gp}{G^{\textrm{pack}}_{A, \J}}
\newcommand{\J}{\mathcal{J}}
\newcommand{\proj}{\textrm{proj}}
\newcommand{\Z}{\mathbb{Z}}

\newcommand{\Pp}{P^{\mathcal{V},\textrm{pack}}}
\newcommand{\pack}{\textrm{pack}}

\ifMP
	\newcounter{mynotes}
	\newcommand{\mnoteless}[1]{}
	\newcommand{\mnote}[1]{}
	\newcommand{\mnoter}[1]{}
	\newcommand{\snote}[1]{}
\else
	\newcounter{mynotes}
	\newcommand{\mnoteless}[1]{\addtocounter{mynotes}{1}{\textcolor{green}{$^{\arabic{mynotes}}$}}%
	\marginpar{\scriptsize \textcolor{green}{ {\arabic{mynotes}.\ {\sf {#1}}}}}}
	\newcommand{\mnote}[1]{\addtocounter{mynotes}{1}{\textcolor{blue}{$^{\arabic{mynotes}}$}}%
	\marginpar{\scriptsize \textcolor{blue}{ {\arabic{mynotes}.\ {\sf {#1}}}}}}
	\newcommand{\mnoter}[1]{\addtocounter{mynotes}{1}{\textcolor{red}{$^{\arabic{mynotes}}$}}%
	\marginpar{\scriptsize \textcolor{red}{ {\arabic{mynotes}.\ {\sf {#1}}}}}}
	\newcommand{\snote}[1]{\addtocounter{mynotes}{1}{\textcolor{red}{$^{\arabic{mynotes}}$}}%
	\marginpar{\scriptsize \textcolor{purple}{ {\arabic{mynotes}.\ {\sf {#1}}}}}}
\fi
\def \myspace {-5pt}%{-10pt}

\title{Theoretical \red{challenges} towards cutting-plane selection}

\ifMP
	\journalname{Mathematical Programming B}

	\author{Santanu~S.~Dey \and Marco~Molinaro}
	
	\institute{
	Santanu~S.~Dey \at
	School of Industrial and Systems Engineering,
	Georgia Institute of Technology \\
	\email{santanu.dey@isye.gatech.edu}
	\\
	Marco~Molinaro \at
	Computer Science Department,
	PUC-Rio \\
	\email{mmolinaro@inf.puc-rio.br}
	}
	
	\date{Received: date / Accepted: date}
\else
	\author[1]{Santanu S. Dey\thanks{santanu.dey@isye.gatech.edu}}
	\author[2]{Marco Molinaro\thanks{molinaro@inf.puc-rio.br}}
	\affil[1]{\small School of Industrial and Systems Engineering, Georgia Institute of Technology} 
	\affil[2]{\small Computer Science Department, Pontifical Catholic University of Rio de Janeiro}
\fi

\begin{document}
\maketitle
\begin{abstract}
While many classes of cutting-planes are at the disposal of integer programming solvers, our scientific understanding is far from complete with regards to cutting-plane selection, i.e., the task of selecting a portfolio of cutting-planes to be added to the LP relaxation at a given node of the branch-and-bound tree. In this paper we review the different classes of cutting-planes available, known theoretical results about their relative strength, important issues pertaining to cut selection, and discuss some possible new directions to be pursued in order to accomplish cutting-plane selection in a more principled manner. Finally, we review some lines of work that we undertook to provide a preliminary theoretical underpinning for some of the issues related to cut selection.
\end{abstract}

\section{Introduction}\label{sec:intro}
Mixed-integer linear programming (MILP) has become a mature area of research both in theory and practice, with powerful commercial and non-commercial solvers available, like \texttt{CPLEX}, \texttt{GUROBI}, \texttt{XPRESS}, and \texttt{SCIP}, to name a few. These modern state-of-the-art solvers use linear programming (LP) relaxation as the primary work-horse to provide dual bounds, which is further enhanced with the use of cutting-planes: linear inequalities that are valid for all integral feasible solutions but not implied by the LP relaxation. We assume familiarity with basic cutting-plane theory; for an introduction to this subject, we refer to the books~\cite{NemWolBook,IPCCZbook}.

We recount here a well-known anecdote about the complete change in outlook towards the use of general-purpose cutting-planes within solvers. Until the early 90s the research community was unanimous: \begin{quote}``In order to solve integer programs of meaningful sizes, one had to exploit the structure of the underlying combinatorial problem.''~\cite{cornuejols:2007}\end{quote} In particular, general-purpose cutting-planes were seen as mostly of theoretical interest. However, in the mid 90s a breakthrough came with a series of papers by Balas, Ceria, Cornu\'ejols, and Natraj, where they obtained striking computational results using Gomory Mixed Integer cuts and the then newly discovered Lift-and-Project cuts, within the branch-and-bound framework~\cite{balas:ce:co:93,balas:ce:co:na:96}. These papers show the efficacy of general-purpose cutting-planes, when \emph{properly selected and employed}. %This development has since revived the importance of understanding the functioning of general purpose cutting planes. See~\cite{Lodi2009} for an excellent survey of the evolution of MIP solvers. 

Today, even more families of cutting-planes are known. While we do not attempt to survey the vast literature on cutting-planes, we discuss the broad perspectives from which they are derived, with some representative examples.
\vspace{\myspace} \newline \newline 
%\begin{itemize}
\noindent \textbf{Geometric cuts.} \red{A set $Q \subseteq\mathbb{R}^n$ is called \emph{lattice-free} if $\textup{int}(Q) \cap \mathbb{Z}^n = \emptyset$. Given a mixed-integer set $P \cap (\mathbb{Z}^n \times \mathbb{R}^q)$ and a lattice-free set $Q$, observe that $$P \cap (\mathbb{Z}^n \times \mathbb{R}^q) \subseteq P \setminus (\textup{int}(Q) \times  \mathbb{R}^q), $$ and therefore valid inequalities for $P \setminus (\textup{int}(Q) \times  \mathbb{R}^q)$ are valid for  $P \cap (\mathbb{Z}^n \times \mathbb{R}^q)$. We call such cuts \emph{geometric cuts}.} \red{A set $Q$ is a \emph{maximal lattice-free convex set} if it is lattice-free, convex, and no other lattice-free convex set properly contains it. A beautiful structural result is that all maximal lattice-free convex sets are polyhedral~\cite{lovasz:1989,basu:co:co:za:2009a}. Note that when $Q$ is a polyhedral set $ \{x\in \mathbb{R}^n\,|\, (a^i)^\top x \leq b^i,~i \in [m]\}$ one can generate a geometric cut by ensuring its validity for all the \emph{disjunctions} identified with the facets of $Q$, i.e., $P \cap \{(x,y) \in \mathbb{R}^{n + q}\,|\, (a^i)^\top x \geq b^i \}$ for $i \in [m]$.}

%These cuts are derived based on the notion of lattice-free sets (sets that do not contain integer points in their interior), and 
\red{Famous geometric cuts} include the \emph{Chv\'atal-Gomory (CG)} cuts~\cite{Gomory58,caprara:fi:1996,letchford:lo:2002b,DBLP:journals/mp/DeyRLM10}, \emph{implied bounds}~\cite{hoffman:pa:1991}, and more generally \emph{split cuts}~\cite{balas:1979,cook:ka:sc:1990,balas:pe:2002}. \red{For all these cuts, the underlying lattice-free set is the so-called split set: $\{x \in \mathbb{R}^n \,|\, \pi_0 \leq \pi^\top x \leq \pi_0 + 1\}$, where $\pi_0 \in \mathbb{Z}$ and $\pi \in \mathbb{Z}^n$.  More general geometric cuts} include \emph{cuts from maximal lattice-free convex sets}~\cite{andersen:lo:we:wo:2007,borozan:2007} (see the excellent survey~\cite{basu2015geometric})  and the closely related \emph{intersection cuts} developed by Balas in 1970s~\cite{balas:1971}. Geometric cuts called \emph{multi-branch split cuts}~\cite{li:ri:2008,DashDG12,DashG13} \red{are based on non-convex lattice-free sets obtained as the union of split sets.} 
\vspace{\myspace} \newline \newline 
\noindent \textbf{Structured relaxations.} \red{Given a mixed-integer set $P \cap (\mathbb{Z}^n \times \mathbb{R}^q)$, the basic idea is to construct a relaxation of $P$, say $R \supseteq P$, and analyze the mixed-integer set $R \cap (\mathbb{Z}^n \times \mathbb{R}^q)$. Valid inequalities for the relaxation then serve as cutting-planes
for the original set, and the additional structure of $R$ eases the search for facets or strong inequalities. Two excellent review articles discussing these cutting-planes are~\cite{marchand:ma:we:wo:2002,johnson:ne:sa:2000}.}
%The idea here is to obtain cuts for (or even a full description of) a more structured relaxation of an MILP set (see the excellent reviews~\cite{marchand:ma:we:wo:2002,johnson:ne:sa:2000}). 

\red{One of the most important families of structured relaxations are the \emph{single constraint relaxations}, i.e., relaxations described by one non-trivial constraint together with various variable bounds and integrality restrictions.} \red{Given a mixed-integer linear set, a single constraint relaxation can be obtained by taking a positive combination of the original constraints. We discuss this in more detail in Section~\ref{sec:agg}.} Single constraint relaxations include \emph{knapsack constraint with binary and continuous variables}~\cite{wolsey:1975,balas:1975,zemel:1978,balas:ze:1984,crowder:jo:pa:1983,carr2000strengthening,hammer1975facet,weismantel19970,van1986valid} (which yield the much used \emph{lifted cover inequalities}),
%\textcolor{purple}{Is this ok?}, 
and \emph{knapsack constraints with general integer and continuous variables}~\cite{atamturk:2003,atamturk:2004,richard:li:mi:2009,atamturk2010mingling,dash2010two}, of which a special case is the \emph{MIR set}~\cite{nemhauser:wo:1990,wolseybook}. \red{As a concrete example, the \emph{mixed-integer rounding (MIR)} set is defined as $\{(x, y) \in \mathbb{Z} \times \mathbb{R}_+\,|\, x + y \geq b \}$. The only facet missing in this description is the \emph{MIR inequality} $$x + \frac{1}{f}\cdot y \geq \lceil b \rceil, $$ where $f = b - \lfloor b\rfloor > 0 .$ The general MIR inequality can be obtained by embedding a general mixed-integer set into the MIR set through column and constraint aggregations, see~\cite{nato} for details.}

%\red{Two important inequalities valid for the above relaxations are the \emph{lifted cover inequalities} and \emph{MIR inequalities}.} \red{The MIR set is given as: $\{(x, y) \in \mathbb{Z} \times \mathbb{R}_+\,|\, x + y \geq b \}$ and  the valid inequality for this set}

\red{\emph{Multiple constraint relaxations} have also been studied, i.e., when the relaxation is described using multiple non-trivial constraints.} Multiple constraint relaxations include those studied in~\cite{agra:co:2007} and the \emph{mixing set}~\cite{gunluk:po:2001,dey:wo:mixing:2010,dash2009mixing,sanjeevi2012mixed} (both generalizing the MIR set), \emph{stable sets}~\cite{johnson:pa:1982,atamturk:ne:sa:2000} (which yield the useful \emph{clique inequalities} applied on conflict graphs by solvers), \emph{fixed-charge networks}~\cite{wolsey1989submodularity,padberg:va:wo:1985,gu:ne:sa:1999,atamturk2017path} (used to obtain \emph{flow cover, path cover, and pack inequalities}), and the \emph{corner and group relaxations}~\cite{gomory:1969,gomory:jo:1972a,gomory:jo:1972b,johnson:1974,johnson:1981}. The most important cuts obtained via the latter are the \emph{Gomory mixed-integer cuts (GMIC)}~\cite{gomory:1960b}; also see \cite{cornuejols:li:va:2003,andersen:co:li:2005} for important variants of GMICs. \red{Given an equalitiy constraint of the form $$\sum_{j= 1}^n a_j x_j  + \sum_{j = 1}^q g_j y_j  = b $$ where $x_j \in \mathbb{Z}_{+}$ and $y_j \in \mathbb{R}_{+}$, the GMIC inequality is given by: $$\sum_{ j =1}^n \textup{min} \left\{ \frac{f_j}{f_0}, \frac{1 - f_j}{ 1 - f_0}\right\}x_j + \sum_{j = 1}^q \textup{max}\left\{ \frac{g_j }{f_0}, \frac{ - g_j}{ 1- f_0} \right\}y_j \geq 1, $$ where $f_j = a_j - \lfloor a_j \rfloor$ and $f_0 = b - \lfloor b \rfloor$. }

The study of the group relaxation has been very popular in the last 10-15 years ~\cite{gomory:jo:2003,gomory:jo:ev:2003,araoz:ev:go:jo:2003,dash2006valid,miller2008new,dey:ri:2008,balas:je:1980,dey:wo:2008a,co:co:za:2009b,basu:co:ko:2011,averkov2015lifting,basu2015operations,dey2010relations,averkov2015lifting,koppe2015electronic,koppe2017new}. Many beautiful and non-trivial structural results have been proved regarding valid inequalities for these relaxations~\cite{basu2013k+1,basu2014equivariant,basu2016minimal,deywolseysfree,YildizC16}. See~\cite{RichardDey,basu2015geometric,basu2016light1,basu2016light2} for reviews on the group relaxation and its variants.
\vspace{\myspace} \newline \newline  				
\noindent \textbf{Subadditive cutting-planes.} \red{An important concept that was first discovered in relation to the group relaxation is the notion of \emph{subadditive cutting-planes}~\cite{johnson:1974}. We state it here for a general pure-integer program: Consider a set $S:= \{x \in \mathbb{Z}^n_+ \,|\, Ax \geq b \}$, where $A \in \mathbb{Q}^{m \times n}$, and a function $f: \mathbb{R}^m \rightarrow \mathbb{R}$ that satisfies the following: 
\begin{enumerate}
\item $f$ is \emph{subadditive}, i.e., $f(u) + f(v) \geq f(u + v)$ for all $u,v \in \mathbb{R}^m$,
\item $f$ is \emph{non-decreasing}, i.e., $f(u) \geq f(v)$ for all $u - v \in \mathbb{R}^m_+$,
\item $f(\textbf{0}) = 0$.
\end{enumerate}  
It is straightforward to verify that $\sum_{j = 1}^n f(A^j)x_j \geq f(b)$ is a valid cutting-plane for $S$. More importanty, it can be shown that the convex hull of $S$ can be obtained by such valid inequalities. 
See~\cite{bell:sh:1977,jeroslow1978cutting,jeroslow1979minimal,johnson:1979,guzelsoy2007duality,moran2012strong} for generalizations to mixed-integer sets and other extensions. 
}
\vspace{\myspace} \newline \newline  				
\noindent \textbf{Convexification via extended formulations and algebraic properties.} The key idea here is to go to a higher dimension, usually the space of products of variables, and then use algebraic properties of the matrix of these variables to imply non-trivial inequalities in the original space. Some of these methods include generating standard cuts like split cuts for extended formulations~\cite{bonami2014cut,bodur2017cutting}, the LP based hierarchies \emph{Reformulation-linearization technique (RLT)}~\cite{sherali:ad:1990,sherali:ad:1994,sherali:sm:ad:2000}, \emph{Lift-and-project cuts (L\&P)}~\cite{balas:ce:co:93,balas:pe:2002}, the semidefinite programming hierarchies due to Lov\'asz and Schrijver \emph{($N_{+}$)}~\cite{Lovasz91}, the \emph{Bienstock and Zuckerberg (BZ) hierarchy}~\cite{bienstock2004subset}, and the more general 
%(applicable to polynomial optimization) 
\emph{Lasserre hierarchy}~\cite{lasserre2001global}.  \red{See the excellent article by Rothvo{\ss}~\cite{rothvoss2013lasserre} on applications of these cuts to develop approximation algorithms.}

\subsection{Theoretical analysis of the strength of cutting-planes}\label{sec:cutstr}

Given this wealth of cuts at our disposal, a crucial question is that of \emph{cutting-plane selection}: Which cuts should a solver employ at a given point in time? Motivated by this issue, there is a large literature on theoretical analysis of cutting-planes trying to better understand the strength of different families of cuts. While the overall goal is be able to solve MILP models efficiently, given the complexity of solvers one cannot expect to be able to perform a theoretical analysis of this performance metric directly, and thus proxy measures are required. In fact, many notions of usefulness or strength of cuts have been proposed. We briefly survey known results in this area. 
\vspace{\myspace} \newline \newline 
\textbf{Finite Cutting-plane Algorithms.} One natural question is: can we solve an MILP in finite time using solely cuts from a given family without branching or any other operation? Gomory gave the first finitely convergent  cutting-plane algorithm for pure integer programs~\cite{Gomory58}, using CG cuts. There have been very few results since then on finite cutting-plane algorithms, see~\cite{neto} and references therein. See~\cite{ChandrasekaranV16} for a polynomial-time cutting-plane algorithm for matching, \cite{ChenKS11} for a finite cutting-plane algorithm for bounded MILPs, and~\cite{DashG13} for a finite cutting-plane algorithm for general MILPs. There are also lower bounds on the length of cutting-plane proofs (length of sequence of cuts needed to prove optimality or infeasibility)~\cite{chvatal1973}. For example, generalizing previous work of Pudl\'ak~\cite{pudlak1997}, Dash~\cite{dash:2005} presents exponential lower bounds for proofs via branch-and-cut procedures that use L\&P and CG cuts.

%########################################################

{\color{black}
\vspace{6pt}
\noindent \textbf{Closure inclusion.} Given a polyhedron $P$ and a family $M$ of cutting-planes, we define the \emph{closure} $M(P)$ as the intersection of all cuts from this family valid for the (mixed-)integer solutions in $P$. For any family of cuts the closure $M(P)$ contains the integer hull, and a smaller closure means a stronger approximation of the latter. 

	A fundamental result is that the GMIC closure, the MIR closure, and the split closure are equal~\cite{nemhauser:wo:1990,CornuejolsLi01}. The L\&P closure~\cite{balas:ce:co:93}, defined only for binary MILPs, is weaker than these other closures. Marchand and Wolsey~\cite{marchand:wo:1999} prove that various well-know classes of cutting-planes are all MIR cuts; thus, the split closure is stronger than the closures of these classes of cuts. Cornu\'ejols and Li~\cite{CornuejolsLi01} provide comparisons between several basic closures, including the split, L\&P, and RLT closures. The papers~\cite{li:ri:2008,DashG13} show that the $t$-branch closure can be strictly stronger than infinitely many iterations of the $k$-branch closure whenever $t > k$. See also~\cite{DashDG12,DASH2011305,dash2015relative}. Laurent~\cite{laurent2003comparison} compares RLT, $N_{+}$, and the Lasserre closures for 0/1 polytopes and shows that the Lasserre construction provides the tightest relaxations. See \cite{Mastrolilli17} for an augmented version of the Lasserre hierarchy for 0/1 problems using high-degree polynomials that implies the BZ hierarchy, and comparison with the CG closure.

%########################################################

\vspace{6pt}
\noindent \textbf{Rank.} \emph{Rank} is a very popular measure of strength, measuring the number of ``round'' of cuts needed to obtain the integer hull. More precisely, the $i^{th}$ closure $M^i(P)$ is defined inductively by applying the closure operator to the set $M^{i-1}(P)$; the \emph{rank} of $P$ with respect to the family of cuts $M$ is then the smallest integer $i$ such that $M^i(P)$ is the integer hull of $P$. 

Schrijver~\cite{Schrijver80} showed that the CG rank is finite for any rational polyhedron $P$. The paper~\cite{GentileVW06} showed that the rank is finite for polytopes, even when restricting to the subclass of mod-2 CG cuts. Ch\'{v}atal, Cook, and Hartmann~\cite{ChvatalCH89} proved lower bounds on the CG rank for natural formulations of various combinatorial problems.
% such as stable set, set covering, set partitioning, knapsack polytope, max cut, and travelling salesman problem. 
It was shown in~\cite{cook1990cutting} that for polyhedra with no integer points the CG rank can be upper bounded by a function of the dimension of the ambient space. For polyhedra in $[0, 1]^n$, Eisenbrand and Schulz showed that the CG rank is always upper bounded by $\mathcal{O}(n^2 \log n)$. Recently, Rothvo{\ss} and Sanit{\'{a}}~\cite{RothvossS17}, showed that there are binary knapsack instances with CG rank $\Omega(n^2)$. Also see \cite{PokuttaS11,pokutta2010rank}. \red{The paper~\cite{cornuejols2016some} provides conditions for a binary IP, in terms of the subgraph induced by the infeasible 0/1  vertices in the skeleton graph of the 0/1-cube, that guarantee small CG rank. See~\cite{benchetrit2018characterizing} for generalizations that characterize polytopes in the 0/1-cube with bounded CG Rank.}

	The L\&P closure~\cite{balas:ce:co:93}, defined only for binary MILPs, has rank at most $n$, the number of binary variables; since this closure is weaker than the GMIC, MIR, and split closures, this yields the same rank upper bound for these cuts on binary MILPs. Moreover, these results are tight, since~\cite{cornuejols:li:2002} gives an example of a binary MILP for which the rank of $n$ is achieved for the split closure. In contrast, for general MILPs~\cite{cook:ka:sc:1990} showed that the rank with respect to split cuts may not be finite, although the $i^{th}$ split closure converges to the integer hull in the limit as $i$ tends to infinity~\cite{Owen2001,del2012convergence}. 

Since the $N_{+}$ operator is also stronger than the L\&P operator, it also has a rank upper bound of $n$ for binary MILPs. Cook and Dash~\cite{cook2001matrix} show that this rank $n$ can be achieved even when $N_{+}$ is combined with CG cuts. See~\cite{kurpisz2015hardest} for a characterization of 0/1 instances that have rank $n$ for the Lasserre hierarchy. Au and Tuncel~\cite{au2016comprehensive,au2016elementary} prove a comprehensive set of results on the rank of the $N_{+}$, BZ, and Lasserre closures and extensions. Also see~\cite{cheung2007computation} for related results.

The lattice-free cut approach and the subbadditive dual approach can be viewed as ``dual'' of each other, see ~\cite{deywolseysfree,conforti2014cut} for concrete results along these lines. Indeed the GMIC cut is the subadditive functional version of the geometric cut from a split set. Therefore, it is expected that even for general MILPs lattice-free cuts should have finite rank, since subadditive functions are known to yield the convex hull.
%An \emph{integral lattice-free polyhedron} $P: = \{x\,|\, (a^i)^{\top}x \leq b_i, \ i \in \{1, \dots, m\}$ is a polyhedron whose vertices are integral and $\textup{interior}(P) \cap \mathbb{Z}^n) = \emptyset$. Therefore, $\cup_{i = 1}^m \{y \in \mathbb{Z}^n\,|\ (a^i)^{\top}x \geq b_i\} = \mathbb{Z}^n$ and can be used to derive disjunctive cutting-planes. 
 An \emph{integral lattice-free polyhedron} is a lattice-free polyhedron whose vertices are integral; \red{the collection of these lattice-free sets is a proper subset of all possible lattice-free convex sets.} Del Pia and Weismantel~\cite{del2012convergence} prove that the rank of general MILPs with respect to integral lattice-free disjunctive cuts is finite. 

\red{If the integrality gaps of an instance is high, we expect that it is difficult to solve this instance, and in particular we expect that more rounds of cutting-planes are needed to obtain the integer hull. The papers~\cite{pokutta2011200,Bodur2017,bodur2017lower}, present lower bounds on rank of CG, split cuts, lattice-free cuts, and closures of cut from multiple-constraint relaxations that increase as the integrality gap increases for packing and covering integer programs (see definition below). }

One measure to compare a cut against split cuts is the so-called \emph{split rank}. This is defined to be the smallest integer $i$ such that the cut under study is valid for the $i^{th}$ split closure. The paper~\cite{dey:lowerbnd:2009} gave a general construction of lower bounds on the split rank of intersection cuts. The papers~\cite{DeyL11,BasuCM12} analyze the split rank of the lattice-free cuts. \red{In fact, given any family of lattice-free convex cuts, one can define the rank with respect to this family. The paper~\cite{del2012rank} characterizes when this rank is finite.}

See~\cite{singh2010improving} for examples showing when low rank CG cuts close the integrality gap significantly and compare these CG cuts to the performance of RLT hierarchy. 

%#########################################################

\vspace{6pt}
\noindent\textbf{Blow-up measure.} One can develop a more refined measure of strength than closure inclusion for packing and covering sets. We say that $P \subseteq \R^n_+$ is a \emph{covering set} if it is upwards closed, namely if $x \in P$ and $y \ge x$, then $y \in P$. Given two covering sets $Q \supseteq  P$ (so $Q$ is a relaxation of $P$), we say that $Q$ is an \emph{$\alpha$-approximation} of $P$ if $P \supseteq \alpha Q := \{\alpha x \mid x \in Q\}$, i.e., pushing the bigger set up by an $\alpha$ factor makes it be contained in the smaller set. Equivalently, $Q$ is an $\alpha$-approximation of $P$ if for all nonnegative objective functions $c \in \R^n_+$ 
	\begin{align}
		\min\{c^\top x \mid x \in Q\} \ge \frac{1}{\alpha} \cdot \min\{c^\top x \mid x \in P\}, \label{eq:blowupIntro} 
	\end{align}
	see~\cite{Goe95} and Lemma 23 of~\cite{molinaro2013understanding}.	This measure can be defined for \emph{packing sets} in an analogous way, see Section \ref{sec:agg1}.
	
	While originally defined in this context for comparing inequalities for the Traveling Salesman Problem~\cite{Goe95}, this measure prominently appears in the analysis of inequalities for the \emph{continuous group relaxation}, namely MILPs defined by $k$ equations with $k$ free integer variables and nonnegative continuous variables (whose projection onto the continuous variables is a covering set). For the case $k=2$, it is known that all facet-defining inequalities are \red{geometric cuts from three types of lattice-free sets}: split, \emph{lattice-free triangle}, or \emph{lattice-free quadrilateral} inequalities. The paper~\cite{basu:bo:co:ma:2009} shows that the triangle closure (and the quadrilateral closure) is a $2$-approximation of the integer hull, while the split closure can be arbitrarily weak under this measure. See~\cite{AwateCGT15} for improved bounds. However, several computational experiments~\cite{espinoza:2008,basu:bo:co:ma:2010,dey:lo:wo:tr:2010,dash2014computational,louveaux2014algorithm} have indicated that triangle and quadrilateral cuts do not work as well as split cuts in practice. 
%See~\cite{dash2014computational,louveaux2014algorithm} for more recent experiment with these cuts. 
In order to better understand this phenomenon, the paper~\cite{basu2011probabilistic} analyzed uniformly generated MILPs and showed that with high probability both split and triangle closures provide very good approximations for the integer hull, thus triangle cuts do not add much over split cuts. Also see~\cite{PiaWW11,he2011probabilistic}.

	This measure has also been used to analyze the related \emph{corner relaxation}. For example, Averkov, Basu, and Paat~\cite{averkov2017approximation} recently showed that for each natural number $n$, a corner polyhedron with $n$ basic integer variables and an arbitrary number of continuous non-basic variables is approximated up to a constant factor by intersection cuts from lattice-free sets with at most $i$ facets if $i > 2^{n -1}$, and that no such approximation is possible if $i \leq 2^{n-1}$.

%\textbf{Strength of cuts relative to split cuts.} Even though the split closure may not finitely converge to the integer hull of general MILPs (i.e., the rank may not be finite), it still has proven to be computationally very effective~\cite{BalasS08,DashGL10}. It is then natural to study the strength of other cuts relative to split cuts.  

% See the paper~\cite{dash2015relative} on the relative strength of \emph{2-branch split cuts} against the second split closure.

%##########################################################

\vspace{6pt}
\noindent \textbf{Minimality, extremality, and facetness.} Unlike the above measures, these attempt to measure the strength of \emph{individual} cuts. Given a polyhedron $P \subseteq \R^n$, a valid inequality $a^\top x \le b$ is \emph{minimal} if there is no other valid inequality $c^\top x \le d$ that strictly dominates it, i.e., $\{x \mid c^\top x \le d\} \subsetneq \{x \mid a^\top x \le b\}$. An inequality is \emph{facet-defining} (or extreme) if it cannot be strictly dominated by a positive combination of other valid inequalities. Notice that facet-defining inequalities are a subset of minimal inequalities, which are a subset of all valid inequalities, and it is well-known that facet-defining inequalities are necessary and sufficient to describe a polyhedron; thus, these give a hierarchy of importance of inequalities. 

	There is a wealth of results on facet-defining inequalities for specific combinatorial optimization problems, and their use has led to vast computational improvements in the 80s and 90s. These notions are also heavily used in the study of the group relaxations. Gomory and Johnson~\cite{gomory:jo:1972a} characterize the minimal inequalities for the infinite group relaxation in terms of subadditive functions. Also see~\cite{koppe2017notions}. For the infinite group relaxation with 1 equation,~\cite{gomory:jo:1972a} also provides a sufficient condition for an inequality to be extreme, the so-called 2-Slope Theorem~\cite{gomory:jo:1972b}. This was later generalized to an arbitrary number of equations~\cite{Cornuejols2013,basu2013k+1}. \textcolor{purple}%{[Add some of Santanu's work!!]} 
For the case of 2 equations, Basu et al.~\cite{basu2014equivariant,BasuEqui2017} provide algorithms for testing whether a minimal valid function is extreme. 
	
	Note that a polyhedron $P$ has a finite number of facet-defining inequalities and, excluding degenerate cases, infinitely many minimal inequalities; thus, restricting to facet-defining inequalities really narrows down the inequalities under consideration. However, interestingly, for the infinite group relaxation (which is infinite dimensional), it was recently proved that the facet-defining inequalities form a dense set within the set of minimal inequalities~\cite{basu2016minimal,lebair}. This casts doubts on whether being a facet-defining inequality is really useful in this context. See the surveys~\cite{basu2016light1,basu2016light2} for a comprehensive list of results on the infinite group relaxation. 	

%##########################################################

%\vspace{6pt}
%\noindent {\color{purple} \textbf{Other measures.}

%\textcolor{purple}{T-space paper, Amitabh's GMI other measure for infinite group} 
%*******************
 
%\textbf{Rank and other comparison of cutting planes based on extended formulations.}
%}
}

%########################################################
%########################################################

\subsection{\emph{Cutting-plane selection:} the need to ask different questions}\label{sec:motivate}
As described in the sections above, many different classes of general-purpose cutting-planes are known today and a wealth of results have been obtained about their strengths. However, when it comes to cutting-plane use and selection our scientific understanding is far from complete. \red{Indeed, the seminal work of Balas, Ceria, Cornuejols, Natraj~\cite{balas:ce:co:na:96} were the first to obtain good results using the same GMI cuts that were already available in the 60s and tested in multiple papers.}
In order to better appreciate 
%the kind of questions one needs to answer, we discuss some of 
the design challenges in implementing cutting-planes, a quote from~\cite{andreello2007embedding} is fitting: 
\begin{quote}``Branch-and-cut designers often have to face an Hamletic question: to cut or not to cut?''\end{quote} Indeed, adding cuts in a na\"ive way could well reduce the number of branching nodes, but the overall computing time may increase dramatically, for example, due to an increase in time to solve the LPs in the branch-and-bound tree. \red{Extreme care is needed regarding how cuts are added. To quote~\cite{balas:ce:co:na:96}, one of the reason of their success is 
\begin{quote}
``The cuts are used selectively, i.e., not all cuts from the pool are part of the formulation solved at every
iteration. This helps us avoid the "clogging" of the linear programs with constraints that are locally
useless."
\end{quote}}

Modern solvers break down the problem of cut selection in two steps. In the Step 1, they maintain a list of valid cutting-planes called \emph{cut-pool} that contains cuts that separate the current fractional point, together with all previously non-discarded cuts (but not added to LP)~\cite{andreello2007embedding,achterberg2009scip}. Then in Step 2 the solver ranks the cuts from this pool and adds a few of them to the LP. As we illustrate next, several issues need to be considered in this selection process. We also point out why unfortunately the traditional analyses of strength of cuts offer only limited help in understanding and addressing these issues.
%Some criteria used to select cuts by solvers and related open questions are:
%\begin{itemize}
%
\vspace{\myspace} \newline \newline
\textbf{Measuring the effectiveness of cutting-planes.} While most theoretical analyses focus on the strength of a whole family of cuts, it is clear that solvers cannot add all of these cuts. Thus, it is important to have measures of the potential effectiveness of individual (or a small set of) cuts as a guide for selection. A common measure of effectiveness of a cut $\alpha^{\top}x \leq \beta$ separating a fractional point $x^{*}$ is the \emph{depth-of-cut}, i.e., $\frac{\alpha^{\top} x^{*} - \beta}{||\alpha||_2}.$ This measure has its drawbacks, and while counter-intuitive, it may rank higher weaker cuts: suppose the problem has non-negativity constraints and $x^*_i = 0$; to obtain larger depth-of-cut one should try to set $\alpha_i = 0$, while the same cut with $\alpha_i > 0$ dominates the latter~\cite{andreello2007embedding} (see~\cite{amaldi2014coordinated} for concrete examples).   
Another potential measure is the volume cut off of a given LP~\cite{gomory:jo:2003}. %See~\cite{ko1997volume} for an example of measure volume of polytopes in this context. 
The recent paper~\cite{basu2017optimal} proves that, for the group relaxation, the depth-of-cut measure may select very weak cuts, while a volume measure is able to single out GMICs, arguably one of the strongest cuts in practice, as the best one. It would be very interesting to better understand cuts with large volume measure and to be able to separate them efficiently. %Also, it would be identify what makes a measure of strength good. 
\red{Also see~\cite{lodiCGCounting} for an efficiency measure based on counting integer points, and~\cite{wesselmann} for other measures of efficiency.}
\vspace{\myspace} \newline\newline
\textbf{Cuts from multiple \red{sources}.} Many solvers walk on the optimal face of the LP relaxation and collect cuts from different bases. (These multiple bases can also be helpful for primal heuristics.) One inspiration for this is how Gomory's finite cutting-plane algorithm makes progress by incrementally cutting off the optimal face~\cite{tobiasMIPTalk,zanette2011lexicography,tobiasAussoisTalk}. This may also help with obtaining stronger cuts; for example, the paper~\cite{cornuejols2012tight} shows how different corner relaxations for the Stable Set problem have very different strengths. Moreover, it is known that the split closure can be obtained from generating split cuts from multiple bases~\cite{Andersen05} (see~\cite{balas:pe:2002,dash2010heuristic,fischetti2011relax} for computational uses of this idea). \red{The paper~\cite{balas:ce:co:na:96} also shows how cuts from different parts of a branch and bound tree can be used in a different part of the branch-and-bound tree for binary MILPs.} However, access to multiple \red{sources of cuts} leads to many questions: Should the cuts selected separate all the known optimal bases? What happens if none of the cuts separates all the known bases? How to adapt the measures of effectiveness to incorporate information of two or more fractional points?
\vspace{\myspace} \newline\newline
\textbf{Parallelism between cuts/objective function.} It has been observed that it is very important to use ``cuts that improve the polyhedron in diverse directions''~\cite{BCC96}.
The dot-product between the normalized values of left-hand sides of two cuts is usually used as measure of parallelism.  See the papers~\cite{andreello2007embedding,achterberg2009scip} on how the performance of cutting-plane non-trivially depends on the setting of this parameter. The paper~\cite{amaldi2014coordinated} presents cut generation explicitly optimizing also for parallelism using a bilevel optimization strategy. The paper~\cite{achterberg2009scip} indicates that considering parallelism with respect to the objective function (favoring cuts in similar direction) may be ineffective. \red{See also~\cite{coniglio} for an MILP-based strategy for adding in a coordinated way the set of $k$ cuts that provides the largest improvement to the relaxation's objective value.} 
%\textcolor{purple}{Is this a good place for this ref??}
To the best of our knowledge, there are no theoretical results/understanding pertaining to parallelism.
\vspace{\myspace} \newline \newline 
\textbf{How many cuts to add.} A natural design parameter is the decision on how many cuts to add in each round. \red{The importance of this parameter is gauged by the success of the the results in~\cite{balas:ce:co:na:96}. In particular, \cite{balas:ce:co:na:96} prescribed adding a large number of GMIC simultaneously obtained from different rows of a simplex tableau corresponding to fractional basic variables. Previous to~\cite{balas:ce:co:na:96}, most implementations added one cut at a time, which lead to very poor convergence. Also~\cite{balas:ce:co:na:96} observes that adding multiple cuts in rounds leads to less parallel cuts.} Some options used are setting a fixed upper bound on the maximum number to be generated~\cite{tobiasThesis} and making this upper bound a fixed parameter times the number of constraints in the original formulation~\cite{andreello2007embedding}. Whether these are the best ways to decide on the number of cuts to be added is not clear. For example, one may expect the number of integer variables to play a role, but to the best of our knowledge no such study has been conducted.  
\vspace{\myspace} \newline\newline
\textbf{\red{Whether to use facet-defining or simply minimal inequalities of a relaxation.}} Given a relaxation whose valid inequalities will be used as cutting-planes, we may choose to restrict our search for cutting-planes to only facet-defining inequalities or use any valid inequality. If we do not use only facet-defining inequalities of the relaxation, a natural option is to use a cut measure, such as violation, for separating valid inequalities from it. For some families of cutting-planes, such as L\&P, it is possible to write the separation problem as an LP. An important design choice is the selection of normalization used in the resulting LP; see for example~\cite{bonami2005using,balas2007new,chen2012computational}. In a recent paper, Wolsey and Conforti~\cite{conforti2016facet} address the problem of efficiently finding an inequality that is violated and either defines an improper face or a facet by the proper use of normalization. They also provide some evidence that this method works on structured and unstructured problems. \red{Cadoux~\cite{Cadoux2010} proposes a projection-based method for generating a set of inequalities that are facet defining for the relaxation and that together dominate the separating inequality with largest depth-of-cut (see also~\cite{CornuejolsLemarechal}).}

\red{As mentioned above}, other recent papers~\cite{basu2016minimal,lebair} considered the question of separating faces versus facet-defining inequalities in the context of the (infinite) group relaxation. Surprisingly, it was shown that for any minimal inequality, there exists an extreme inequality that approximates the minimal inequality as closely as desired. Thus, in this context separating faces is effectively ``equal to'' separating facets.
%\textcolor{purple}{Repeating too much of above, see section on Minimality, Extremality, Facetness.} 
While some special cases are understood, many open problems remain regarding the selection of cuts from relaxations.
%\mnoteless{For the revision, one possibility is to move this ``density paper'' to the ``measures of effectiveness'' section} 
\vspace{\myspace} \newline\newline
\textbf{Cut sparsity.} It is well established that LP solvers perform much better with sparse formulations. Therefore, it is important to add sparser cutting-planes to maintain sparsity of the LP. We expand on the discussion pertaining to sparsity in Section~\ref{sec:sparsity}. 
\vspace{\myspace} \newline\newline
\textbf{Numerical properties, validity, and stability.} There are two main issues related to numerical properties of cuts. The first is that of validity: because most implementations use finite-precision arithmetic, round off errors may induce the generation and use of invalid cuts~\cite{margotSafe,margotGerardSafe}. This has motivated the line of research on \emph{numerically safe computations} in integer programs; see for example~\cite{cookSafe} for the numerically safe MIR/GMI cuts. An alternative method is to have solvers using exact rational arithmetic~\cite{Cook2013,SCIP5}. The second issue is that of numerical stability (e.g., ill conditioning of LPs), which is particularly salient when multiple rounds of cuts are generated~\cite{zanette2011lexicography}. To avoid this, typically cuts with large \emph{dynamism}, i.e., ratio of largest to smallest absolute value of coefficients, are discarded~\cite{fischetti2011relax}. Perhaps more interestingly, there have been a few proposed methods for generating cuts in a branch-and-cut framework to mitigate this issue, such as the use of rank-1 cuts~\cite{dash2010heuristic,Bonami2012}, the \emph{relax-and-cut} framework~\cite{fischetti2011relax}, using a simplified version of the lexicographic simplex~\cite{zanette2011lexicography}, and the use of \emph{generalized intersection cuts}~\cite{Balas2013}. More research would be welcome to better understand the source of instability introduced by cut generation and how to avoid it in an efficient way.
\vspace{\myspace} \newline\newline
\textbf{Strength of cutting-planes with respect to problem being solved.
%\mnote{Another possible title is ``Interaction between cuts and problem structure''}
} 
It is well-known that clique cuts are extremely useful for the stable set IPs, and flow covers are extremely important for certain fixed-charge network flow sub-structures~\cite{AchterbergR10}. However, very few results of this nature are known, i.e., that connect the performance of a particular general purpose cut to broad sub-classes of problems. Indeed, much of the results on the strength of cutting-planes, as discussed in Section \ref{sec:cutstr}, is via worst-case results.
%, e.g., the worst case upper bound on rank or a particular instance that achieves a lower bound on the rank. 
Such results only provide a coarse indication of which families of cuts seem to be strong, while not indicating which class of cutting-planes is good (or, equally importantly, bad) for a given class of instances. 
Recently, the notion of reverse rank was introduced and studied (for CG and split cuts)~\cite{ConfortiPSFGCG15,ConfortiPSFSpli15} which connects instances to cut strength. In particular, the question asked is the following: for which problems, can there exist an LP relaxation with arbitrarily bad rank for a given class of cuts? While this is an important first step, many questions in this direction remain to be explored. 

{\color{black}
\vspace{6pt}
\noindent \textbf{Balancing the different goals.} The points above illustrate that many, possibly conflicting, goals should be taken into account when generating cuts; the question is then how to balance these goals. Overall, the common strategy is to: 1) first generate a collection of violated cuts from different families using heuristics that typically aim at obtaining cuts with maximum effectiveness; 2) discard cuts whose performance on one of the goals (e.g., parallelism) is worse than a given  (possibly dynamic) amount (possibly skipping this step for cuts with particularly high ``quality''); 3) rank the remaining cuts using a linear combination of the different goals, and add the top cuts~\cite{andreello2007embedding,achterberg2009scip,wesselmann}. To the best of our knowledge, there are very few works that during cut generation explicitly attempt to optimize on more than one goal simultaneously. Here we can cite~\cite{amaldi2014coordinated}, which simultaneously optimizes efficiency and parallelism by solving a bilevel optimization problem, and~\cite{Cadoux2010} (mentioned above) that generates cuts that are facets for a relaxation and also dominate the separating cut with largest depth-of-cut.
}

%We review some new results connecting performance of large classes of cuts on packing and covering instances in Section~\ref{sec:agg}.
%\end{itemize}
%For instance, many results focus on worst-case, etc., though they do provide a coarse indication of which \textbf{families} of cuts seem to be strong.  
\vspace{8pt}
%Modern solvers have sophisticated heuristics for selection of cutting-planes to use from the cut-pool to the LP. 
We hope that the above discussion illustrates that there are many important issues pertaining to cut generation and selection, and also several exciting directions that can lead to a better understanding of cutting-plane selection. In the next two sections we review some lines of work that we undertook to provide a preliminary theoretical underpinning for some of the issues appearing in cut-selection discussed above, namely sparsity in Section~\ref{sec:sparsity} and use of knapsack relaxation cuts for packing and covering instances (i.e. connecting some family of cut to important sub-family of instances) in Section~\ref{sec:agg}.

%, but as the above indicates, when it comes to cutting-plane use and selection, our scientific understanding is far from complete. The reason	
%##########################################################
%##########################################################
%##########################################################
%##########################################################
%##########################################################
%##########################################################

\section{Sparsity}\label{sec:sparsity}
%	\blue{[Dump from sparsity paper and one of your grants :)]}
	Sparsity is a very important topic in all areas of scientific computing. Consider the following statistics: \emph{The average number (median) of {non-zero entries in the constraint matrix} of MIPLIB 2010~\cite{MIPLIB2010} instances is {\textbf{$1.63\%$}} ({\textbf{$0.17\%$}}).
}
%\begin{center}
%\fbox{
%\begin{minipage}{0.85\linewidth}
%The average number (median) of \textbf{non-zero entries in the constraint matrix} of MIPLIB 2010~\cite{MIPLIB2010}  instances is \underline{\textbf{$1.63\%$}} (\underline{\textbf{$0.17\%$}}).
%\end{minipage}
%}
%\end{center}

Sparse LPs can be solved efficiently~\cite[Chapter~3]{Coleman84}~\cite{pggthesis,Gill:mu:sa:wr}.  An IP solver typically solves a large number of LPs in a \emph{branch-and-bound tree} to solve one IP instance. Therefore, clearly one of the biggest benefits of sparsity for state-of-the-art IP solvers is the sparsity of the underlying LP. In a very revealing study~\cite{Walter14}, the authors conducted the following experiment: They added a very dense valid equality constraint to a few constraints of the LP relaxation at each node while solving IP instances from MIPLIB using CPLEX. This does not change the underlying polyhedron, but makes the constraints dense. They observed approximately $25\%$ increase in time to solve the instances if $9$ constraints were made artificially dense.
	
In order to keep the underlying LP sparse, most commercial MILPs solvers consider sparsity of cuts as an important criterion for cutting-plane selection and use. The use of sparse cutting-planes may be viewed as a compromise between two competing objectives. As discussed above, on the one hand, the use of sparse cutting-planes aids in solving the linear programs encountered in the branch-$\&$-bound tree faster. On the other hand, it is possible that `important' facet-defining or valid inequalities for the convex hull of the feasible solutions are dense and thus without adding these cuts, one may not be able to attain significant integrality gap closure. This may lead to a larger branch-$\&$-bound tree and thus result in the solution time to increase. 

It is challenging to simultaneously study both the competing objectives in relation to cutting-plane sparsity. Therefore, a first approach to understanding usage of sparse cutting-planes is the following:  How much do we lose in the closure of integrality gap if we only use (some of the) sparse cuts (as against dense cuts)?

%##########################################################
%##########################################################

\subsection{Using all sparse cuts}
	
The paper~\cite{deyMolinaroWang:2015} considers the following starting point for understanding better this issue: \emph{If we are able to separate and use \textbf{all} valid inequalities with a given level of sparsity how much does this cost in terms of loss in closure of integrality gap?} 

{Considered more abstractly, consider a polytope $P$ contained in the $[0,\ 1]^n$ hypercube, representing the integer hull of an integer program. (The assumption $P \subseteq [0,1]^n$ is WLOG and just a normalization.)} A cut $a x \le b$ is called \emph{$k$-sparse} if the vector $a$ has at most $k$ nonzero components. For the polytope $P$, define $P^k$ as the best outer-approximation obtained from $k$-sparse cuts, that is, it is the intersection of all $k$-sparse cuts valid for $P$. Consider the following natural measure of quality of approximation of $P$ by sparse inequalities: $\dist(P, P^k):= \max_{x \in P^k} \left(\textup{min}_{y \in P} \| x - y\|\right)$, where $\| \cdot \|$ is the $\ell_2$ norm. Note that $\dist(P, P^k)$ is an upper bound on the depth of cut measure discussed in Section~\ref{sec:motivate} when the left-hand-side of the cut is normalized to be a unit vector. The largest distance between any two points in the $[0,\ 1 ]^n$ hypercube is at most $\sqrt{n}$. Therefore we will compare the value of $\dist(P, P^k)$ to $\sqrt{n}$.

We were able to obtain the following result in~\cite{deyMolinaroWang:2015} to understand how well sparse cuts approximate the actual shape of the integer hull.

\begin{theorem}[Upper Bound on $\dist(P, P^k)$]\label{thm:upperall}
Let $n \geq 2$. Let $P \subseteq [0, 1]^n$ be the convex hull of points $\{p^1, \dots, p^t\}$. Then \red{$$\dist(P, P^k) \leq \textup{min}\left\{ 8\frac{\sqrt{n}}{\sqrt{k}}\sqrt{2\log 4 tn}, 2\sqrt{n}\left(\frac{n}{k} -1\right) \right\}. $$}
\end{theorem}

In particular, consider polytopes with `few' vertices, say $n^q$  vertices for some constant $q$. Suppose we decide to use cutting-planes with half sparsity (i.e. $k = \frac{n}{2}$), a reasonable assumption in practice. Then plugging in these values, it is easily verified that {$\dist(P, P^k) \leq {16}\sqrt{(q + 1)\log n} \approx c \sqrt{\log n}$} for a constant $c$, which is a significantly small quantity in comparison to $\sqrt{n}$. In other words, \emph{if the number of vertices is small, independent of the location of the vertices, using half sparsity cutting-planes allows us to approximate the integer hull very well.} We believe that as the number of vertices increase, the structure of the polytope becomes more important in determining $\dist(\P, \P^k)$ and Theorem \ref{thm:upperall} only captures the worst-case scenario. 
%Theorem \ref{thm:upperall} presents a theoretical justification for the use of sparse cutting-planes in many cases.  

%Empirical experiment in \cite{DeyMW14} confirm that this theoretical result is similar to observation in practice.  
In the paper~\cite{deyMolinaroWang:2015}, it was also shown that the bound presented in Theorem \ref{thm:upperall} is tight for random 0-1 polytopes.
%\mnote{Mention that sparse cuts are bad for random knapsack?}

The proof of Theorem \ref{thm:upperall} uses probabilistic techniques which we outline below.

%	\red{Will work on the proof tomorrow}

%	\red{Need to make sure this actually proves the statement above, in particular check constants!!}

	\begin{proof}[Proof sketch of Theorem \ref{thm:upperall}]
		Consider a polytope $\P = \conv\{p^1, p^2, \ldots, p^t\}$ in $[0,1]^n$. We only prove the first bound in the theorem, that is, we show that $\dist(\P, \P^k)$ is at most $4 \lambda^*$ for \red{$$\lambda^* = \frac{2n^{1/2}}{\sqrt{k}} \sqrt{2\log 4 t n}.$$} Equivalently, we show that every point at distance more than $4 \lambda^*$ away from $\P$ is cut off by a $k$-sparse inequality valid for $P$. Assume that $4 \lambda^*$ is at most $\sqrt{n}$, otherwise the result is trivial.
		
	Let $u \in \R^n$ be a point at distance more than $4 \lambda^*$ away from $\P$. Let $v$ be the closest point in $\P$ to $u$. We can write $u = v + \lambda d$ for some vector $d$ with $\|d\|_2 = 1$ and $\lambda > 4\lambda^*$. The inequality $dx \le dv$ is valid for $\P$, so in particular $d p^i \le dv$ for all $i \in [t]$. In addition, this inequality cuts off $u$: $du = dv + \lambda > dv$. The idea is to use this extra slack factor $\lambda$ in the previous equation to show we can sparsify the inequality $dx \le dv$ while maintaining separation of $\P$ and $u$. It then suffices to prove the following lemma.
		
		\begin{lemma} \label{lemma:existsSep}
			There is a vector $\tilde{d} \in \R^n$ such that: $\tilde{d}$ is $k$-sparse, $\tilde{d} p^i \le \tilde{d} v + \frac{\lambda}{2}$ for all $i \in [t]$, and $\tilde{d} u > \tilde{d} v + \frac{\lambda}{2}$.
%			\begin{enumerate}
%				\vspace{-8pt}
%				\item $\tilde{d}$ is $k$-sparse 
%				\vspace{-8pt}
%				\item $\tilde{d} p^i \le \tilde{d} v + \frac{\lambda}{2}$ for all $i \in [t]$
%				\vspace{-8pt}
%				\item $\tilde{d} u > \tilde{d} v + \frac{\lambda}{2}$.
%			\end{enumerate}
		\end{lemma}
		
		To prove the lemma we construct a random vector $\tilde{\D} \in \R^n$ satisfying these three properties simultaneously with non-zero probability. Let $\alpha = \frac{k}{2 \sqrt{n}}$. \emph{Assume for simplicity that no coordinate of $d$ is too big, that is, $\alpha |d_i| \le 1$ for all $i$; bigger coordinates can be handled with simple modifications of the proof.} Then define $\tilde{\D}$ as the random vector with independent coordinates where $\tilde{\D}_i$ takes value $\sign(d_i)/\alpha$ with probability $\alpha |d_i|$ and takes value $0$ with probability $1 - \alpha |d_i|$. (For convenience we define $\sign(0) = 1$.) 
		
		\paragraph{Property 1:} \emph{The probability that $\tilde{\D}$ is not $k$-sparse is at most $\frac{1}{4n}$.} Note that this vector is sparse in \emph{expectation}, since the expected number of non-zero coordinates is $\alpha \|d\|_1 \le k/2$. A simple calculation also upper bound the variance $\Var(\textrm{\# of non-zero coords. of $\tilde{\D}$}) \le \frac{k}{2}.$ Since the coordinates of this vector are independent, by the Central Limit Theorem we expect it has approximately $\frac{k}{2} \pm \sqrt{\frac{k}{2}} \ll k$ non-zero coordinates. In fact, one can make this quantitative by using Bernstein's Inequality \cite{concentration}, obtaining that the probability that $\tilde{\D}$ is not $k$-sparse is at most $\frac{1}{4n}$ (this uses the assumption $4 \lambda^* \le \sqrt{n}$).
		
		\paragraph{Property 2:} \emph{The probability that for some vertex $p^i$ we have $\tilde{\D} p^i \ge \tilde{\D} v + \frac{\lambda}{2}$ is at most $1/4n$.} Given the ``slack of $\lambda$'' (which is at least $4 \lambda^*$) mentioned before, it suffices to show that $\Pr(\max_{i \in [t]} (\tilde{\D} (p^i - v) - d(p^i - v)) \ge 2  \lambda^*)$ is at most $1/4n$.  Define the centered random variable $\bZ = \tilde{\D} - d$. Triangle inequality and the fact that $v \in \conv\{p^1,p^2,\ldots,p^n\}$ give that $\max_{i \in [t]} \bZ (p^i - v) \le 2 \max_{i \in [t]} |\bZ p^i|$; so we upper bound the probability that the right-hand side is larger than $\lambda^*$. 
		
		First, fix a single $i \in [t]$. Since $\bZ$ is centered we have $\E [\bZ p^i] = 0$. Simple calculations show that $$\Var(\bZ p^i) = \Var(\tilde{\D} p^i) \le \frac{\sqrt{n}}{\alpha} \le \frac{(\lambda^*)^2}{4 \log 4 tn}.$$ Then, again we expect that $\bZ p^i$ is approximately $0 \pm \sqrt{\Var(\bZ p^i)} \ll \lambda^*$, and using Bernstein's inequality one actually obtains $\Pr(|\bZ p^i| \ge \lambda^*) \le \frac{1}{4tn}.$
			
			This analysis only considered a single $i \in [t]$. But employing a union bound we can put all these bounds together and obtain that the probability that \emph{any} $i \in [t]$ has $|\bZ p^i| \ge \lambda^*$ is at most $\frac{t}{4tn} = \frac{1}{4n}$, and thus $\Pr(\max_{i \in [t]} |\bZ p^i| \ge \lambda^*) \le \frac{1}{4n}$ as desired.

		\paragraph{Property 3.} \emph{The probability that $\tilde{\D}u > \tilde{\D}v + \frac{\lambda}{2}$ is at most $\frac{1}{2n - 1}$.} Recalling $u - v = \lambda d$, it is equivalent to upper bound $\Pr(\tilde{\D}d \le 1/2)$. In expectation we have $\E[\tilde{\D}d] = dd = 1$, and simple calculations (using $4 \lambda^* \le \sqrt{n}$) show that for every scenario we have $\tilde{\D}d \le n$. Then employing Markov's inequality to the non-negative random variable $n - \tilde{\D}d$, we get $\Pr(\tilde{\D}d \le 1/2) \le 1 - \frac{1}{2n - 1}$.

		\medskip
		
		To prove Lemma \ref{lemma:existsSep}, simply employ a union bound to see that with positive probability $\tilde{\D}$ satisfies the three properties simultaneously.
	\end{proof}

%#########################################################
%#########################################################

%\subsection{Sparse inequalities: less idealized setting}

One drawback of this result is that it considers a very stylized situation, since almost always some dense cutting-planes are used or one is interested in approximating the integer hull only along certain directions or we may be able to use reformulations. Therefore, the paper~\cite{dey2015some} studies the following questions: (i) Are there polytopes, where the quality of approximation by sparse inequalities cannot be significantly improved by adding polynomial (or even exponential) number of \emph{arbitrary} valid inequalities? (ii) Are there polytopes that are difficult to approximate under \emph{every} rotation? (iii) Are there polytopes that are difficult to approximate in \emph{almost all} directions using sparse inequalities? it was shown that the answer each of the above questions in the positive. This is perhaps not surprising: an indication that sparse inequalities do not always approximate integer hulls well even in the more realistic settings discussed above. %Understanding when sparse inequalities are effective in all the above settings is an important research direction. 

\subsection{Sparse cuts for sparse formulations}\label{sec:sparseforsparse}
	
Another feature missing from the previous results is that they work on worst-case instances. However, as mentioned before, in practice most instances in practice have a \textbf{sparse LP formulation}. %Indeed, we expect that there is an interaction between the sparsity structure of the formulation and sparse cuts. 
For one, it is reasonable to expect that sparse cut perform better on sparse formulations. Moreover, the sparsity structure of the formulation can serve as a guide to choosing the support of the which sparse cuts to use. 
	
As an extreme example, consider the feasible region of the following separable MILP:
\begin{eqnarray*}
\begin{array}{llcl}
A^1x^1& &\leq &b^1\\
&A^2x^2&\leq& b^2\\
x^1\in \mathbb{Z}^{p_1}\times \mathbb{R}^{q_1},&x^2\in \mathbb{Z}^{p_2}\times \mathbb{R}^{q_2}&\end{array}
\end{eqnarray*}
Since the constraints are completely disjoint in the $x^1$ and $x^2$ variables, the convex hull is obtained by adding valid inequalities in the support of the first $p_1 + q_1$ variables and another set of valid inequalities for the second $p_2 + q_2$ variables. Therefore, sparse cutting-planes, in the sense that their support is not on all the variables, are sufficient to obtain the convex hull. Now one would like to extend such a observation even if the constraints are not entirely decomposable, but ``loosely decomposable''. Indeed this is the hypothesis that is mentioned in the classical computational paper~\cite{crowder:jo:pa:1983}. This paper solves fairly large-scale 0-1 integer programs (up to a few thousand variables) within an hour in the early 80s, using various preprocessing techniques and the lifted cover inequalities within a cut-and-branch scheme. To quote from this paper:
%\begin{quote}``All problems are characterized by {sparse} constraint matrix with rational data."
%\end{quote}
\begin{quote}
``the inequalities that we generate preserve the sparsity of the constraint matrix.''
\end{quote}
%Since the constraints matrices are sparse, most of the cuts that are used in this paper are sparse. 

	The paper \cite{deyMolinaroWang:2017} examines this interplay between formulation sparsity and sparse cuts. While the paper considers packing MILPs, covering MILPs, and a more general form of MILPs where the feasible region is arbitrary together with assumptions guaranteeing that the objective function value is non-negative, here we focus only on packing problems. So consider a packing problem of the form: 
	\begin{align}
\textup{max}~& c^Tx \notag \\
s.t. ~&Ax \leq b \notag \\
&x_j \in \mathbb{Z}_+, \forall j \in \mathcal{L}  \label{eq:packingSparse} \tag{P} \\
&x_j \in \mathbb{R}_+, \forall j \in [n]\backslash \mathcal{L}, \notag
\end{align}
with $A \in \mathbb{Q}_{+}^{m \times n}$, $b\in \mathbb{Q}_+^m$, $c \in \mathbb{Q}_+^{n}$ and $\mathcal{L} \subseteq [n]$.

	In order to formalize the notion of ``almost decomposable'', \cite{deyMolinaroWang:2017} uses an interaction graph of the columns of the constraint matrix $A$. 
	
	\begin{definition}[Packing interaction graph of A]\label{defn:packgraph}
Consider a matrix $A \in \mathbb{Q}^{m \times n}$. Let $\mathcal{J}:= \{J_1, J_2 \dots, J_q\}$ be a partition of the index set of columns of $A$ (that is $[n]$). We define the \emph{packing interaction graph} $\Gp = (V,E)$ as follows: (i) There is one node $v_j \in V$ for every part $J_j \in \mathcal{J}$. (ii) For all $v_i, v_j \in V$, there is an edge $(v_i, v_j) \in E$ if and only if there is a row in $A$ with non-zero entries in both parts $J_i$ and $J_j$, namely there are $k \in [m]$, $u \in J_i$ and $w \in J_j$ such that $A_{ku} \neq 0$ and $A_{kw} \neq 0$.
%for $j \in [q]$.
%\begin{enumerate}
%\item There is one node $v_j \in V$ for every part $J_j \in \mathcal{J}$. %for $j \in [q]$.
%\item For all $v_i, v_j \in V$, there is an edge $(v_i, v_j) \in E$ if and only if there is a row in $A$ with non-zero entries in both parts $J_i$ and $J_j$, namely there are $k \in [m]$, $u \in J_i$ and $w \in J_j$ such that $A_{ku} \neq 0$ and $A_{kw} \neq 0$.
%\end{enumerate}
\end{definition}

Now we turn to the control over the sparsity structure of cuts one is willing to use. For that, \cite{deyMolinaroWang:2017} uses a list of allowed supports of the cuts and considers the result of adding all the valid cuts within these supports. 
	
	\begin{definition}[Column block-sparse closure]
Given the problem $\textup{(P)}$, let $\mathcal{J}:= \{J_1, J_2, \dots, J_q\}$ be a partition of the index set of columns of $A$ (that is $[n]$) and consider the packing interaction graph $\Gp = (V, E)$. With slight overload in notation, for a set of nodes $S \subseteq V$ we say that inequality $\alpha x \le \beta$ is a \emph{sparse cut on} $S$ if \red{its support is contained in $\bigcup_{v_j \in S} J_j$.}
%it is a sparse cut on its corresponding variables, namely $\bigcup_{v_j \in S} J_j$. 
The closure of these cuts is denoted by $P^{(S)} := P^{(\bigcup_{v_j \in S} J_j)}$. Given a collection $\mathcal{V}$ of subsets of the vertices $V$ (the \emph{support list}), we use $\Pp$ to denote the closure obtained by adding all sparse cuts on the sets in $\mathcal{V}$'s, namely $$\Pp := \bigcap_{S \in \mathcal{V}} P^{(S)}.$$
%\begin{enumerate}
%	\item With slight overload in notation, for a set of nodes $S \subseteq V$ we say that inequality $\alpha x \le \beta$ is a \emph{sparse cut on} $S$ if it is a sparse cut on its corresponding variables, namely $\bigcup_{v_j \in S} J_j$. The closure of these cuts is denoted by $P^{(S)} := P^{(\bigcup_{v_j \in S} J_j)}$.
%	
%	\item Given a collection $\mathcal{V}$ of subsets of the vertices $V$ (the \emph{support list}), we use $\Pp$ to denote the closure obtained by adding all sparse cuts on the sets in $\mathcal{V}$'s, namely $$\Pp := \bigcap_{S \in \mathcal{V}} P^{(S)}.$$
%\end{enumerate}
\end{definition}
	
\begin{example}[Two-stage stochastic packing program]\label{ex:2stage}
The deterministic equivalent of a two-stage stochastic  program with finitely many realizations of uncertain parameters in the second stage has the following form:
	\begin{align}
		\max ~~~~& \sum_{i = 0}^K (c^i)^T x^i\notag\\
	  s.t. ~~~~& A x^0 \le b \label{eq:stoch} \tag{2-SSP}\\
	    	     & A^i x^0 + B^i x^i \le b^i ~~~~~~~~\forall i \in [k]\notag\\
	    	     &x^0, \ldots, x^k \textrm{ integral}, \notag
	\end{align}
	where $x^0$ are the first stage variables and the $x^i$'s for $i \ge 1$ correspond to each realization in the second stage. Note that there are no constraints involving two different realizations.
	
It is natural to put all the first stage variables $x^0$ into one block and each of the second stage variables $x^i$ (for $i \ge 1$) into a separate block of variables. So we have a graph $\Gp$ with vertex set $\{v_0, v_1, \dots, v_{k}\}$ and edges $(v_0, v_1), (v_0, v_2), \dots, (v_0, v_k)$. 

Consider the cuts on the support of first stages variables together with the variables corresponding to one second stage realization, the so-called \emph{single-scenario cuts}. Such cutting-planes are well-studied, see for example~\cite{Boduretal,Caroephd,Sen2005,ZhangK14}. 
%A standard technique in stochastic integer programming is to make multiple copies of the first stage variables, which are connected through equality constraints, and relax these (``nonanticipativity'') equality constraints (via Lagrangian relaxation methods) to produce computationally strong bound~\cite{CaroeS99}.\mnoteless{Reduce sentence?}
%%\mmnote{See if this is what you mean}. 
%It is straightforward to see that in the case where there is complete recourse, the closure of the single-scenario cuts, or equivalently the natural sparse closure, 
%gives the same bound as this nonanticipativity dual.\mnoteless{Reduce sentence?}
%
%	Another possibility is to consider the (weaker) cuts supported only on each block $x^i$ (for $i=0,\ldots,k$).
\end{example} 

	Intuitively, the ``fewer edges'' the graph $\Gp$ has, the more ``almost decomposable'' the blocks of variables $J_1, J_2 \dots, J_q$ are. However, it is also crucial to capture the \emph{interplay} between this graph structure and the supports of the sparse cuts to be used. For that, \cite{deyMolinaroWang:2017} defines the notion of \emph{mixed stable sets} and \emph{fractional mixed chromatic number} (denoted by $\eta^{\mathcal{V}}$) based on them (we refer to the paper for the exact definition). In the case the support list $\mathcal{V}$ consists of just the individual nodes of $\Gp$ these notions reduce to the standard stable sets and fractional chromatic number. 
	
	With these definitions at hand, the main result of \cite{deyMolinaroWang:2017} for packing problem is a bound on the strength of sparse cuts, \red{where the support of the sparse cuts is} in a given list $\mathcal{V}$, \red{such that the bound} \emph{solely depends on the fractional mixed chromatic number of $\Gp$}. In particular, it is independent of the dimension and the data of the problem. 
	
	\begin{theorem}\label{thm:packing}
	Consider a packing integer program as defined in \eqref{eq:packingSparse}. Let $\mathcal{J} \subseteq 2^{[n]}$ be a partition of the index
set of columns of $A$ and let $G = \Gp = (V,E)$ be the packing interaction graph of $A$. Then for any sparse cut support list $\mathcal{V} \subseteq 2^V$ we have
\begin{eqnarray*}
z^{\mathcal{V}, \pack} \leq \eta^{\mathcal{V}}(G) \cdot z^I,
\end{eqnarray*}
where $z^{\mathcal{V}, \pack}$ is the optimal value over the sparse closure $P^{\mathcal{V},\pack}$ and $z^I$ is the optimal value of the original integer program \eqref{eq:packingSparse}.  
\end{theorem}

	Since the definitions involved in this theorem are a mouthful, we consider the more concrete case \red{of} the 2-stage stochastic packing program \eqref{eq:stoch}. Consider graph $\Gp$ defined as in Example \ref{ex:2stage}, and consider the support list $\mathcal{V}$ consisting of just the individual nodes of $\Gp$. In this case, only $x^i$-cuts are allowed (for $i=0,1,\ldots,k$), i.e., cuts supported only on the $x^i$ variables. As mentioned above, in this case the fractional mixed chromatic number reduces to the standard fractional chromatic number, which equals 2 for the star graph $\Gp$. Therefore, Theorem \ref{thm:packing} reduces to the following. 
	
	\begin{corollary} \label{cor:packing}
		Consider a 2-stage stochastic integer program \eqref{eq:stoch}. If $z^I$ is the optimal value of this integer program and $z^{sparse}$ is the optimal value of the relaxation consisting only of all $x^i$-cuts (for $i=0,\ldots,k$), then $$z^{sparse} \le 2 z^I.$$
	\end{corollary}
	
	We prove this corollary to illustrate the simple ideas behind the the proof of Theorem \ref{thm:packing}. 
	
	\begin{proof}[Proof of Corollary \ref{cor:packing}.]
	Let $P^I$ denote the integer hull of the integer program. Let $\bar{x} = (\bar{x}^0, \ldots, \bar{x}^k)$ be a feasible solution for the relaxation consisting only of all the $x^i$-cuts cuts (for $i = 0,\ldots,k$). We will upper bound its objective value by $2 z^I$ by constructing two good integer solutions in $P^I$. 
	
	Let $\tilde{x}^i$ be the optimal solution just for the $i$th block of the problem, namely 
		\begin{align*}
			\tilde{x}^i \in \argmax \bigg\{ (c^i)^T x^i : (0, \ldots, 0,  \tikzmark{calculator}x^i, 0,\ldots, 0) \in P^I\bigg\}.
		\end{align*}
%%%%%
%\ipMP
%\else
%	\begin{tikzpicture}[overlay,remember picture]
%	    \draw[arrows=->] 
%	    ( $ (pic cs:calculator) +(3pt,-2.5ex) $ ) -- 
%	    ( $ (pic cs:calculator) +(3pt,-0.5ex) $ );
%	    \node[anchor=north]
%	    at ( $ (pic cs:calculator) +(3pt,-2ex) $ )
%	    {\small at the $i$th block};
%	\end{tikzpicture}
%\fi
%%%

\vspace{10pt}\noindent Because we assumed the problem to be of packing-type, this is equivalent to 
		\begin{align}
			\tilde{x}^i \in \argmax \bigg\{ (c^i)^T x^i : x^i \in \proj_{x^i} P^I\bigg\}, \label{eq:proj}
		\end{align}
		where $\proj_{x^i}$ denotes the projection onto the variables $x^i$. Assume without loss of generality that $\tilde{x}^i$ is integral.

		By definition, $(\tilde{x}^0, 0, \ldots, 0)$ is a feasible solution for \eqref{eq:stoch}; moreover, because there are no constraints in \eqref{eq:stoch} connecting any pair of different blocks $x^i,x^j$ for $i,j \ge 1$, we obtain that $(0, \tilde{x}^1, \ldots, \tilde{x}^k)$ is also a feasible solution. (This indicates the importance of stable sets in this context.) Each of these two solutions has value at most the optimum $z^I$, thus 
		\begin{align}
			\sum_{i=0}^{k} (c^i)^T \tilde{x}^i \le 2 z^I. \label{eq:proofStoch1}
		\end{align}
		
		Now we claim that $\bar{x}^i \in \proj_{x^i} P^I$: this is because every valid inequality $\alpha^i x^i \le \beta$ for $\proj_{x^i} P^I$ naturally yields the $x^i$-cut $(0,\ldots,0,\alpha^i,0\ldots,0) x \le \beta$ in the original space, which is satisfied by $\bar{x}$, and hence $\alpha^i \bar{x}^i \le \beta$. The optimality of $\tilde{x}^i$ in \eqref{eq:proj} then gives that $(c^i)^T \bar{x}^i \le (c^i)^T \tilde{x}^i$ (for all $i = 0,\ldots, k$), and so $$\sum_{i = 0}^k (c^i)^T \bar{x}^i \le \sum_{i = 0}^k (c^i)^T \tilde{x}^i.$$ Together with inequality \eqref{eq:proofStoch1} this concludes the proof.
	\end{proof}
	
	The paper~\cite{deyMolinaroWang:2017} also present similar phenomena for covering problems, and for more general MILPs where the feasible region is arbitrary together with assumptions guaranteeing that the objective function value is non-negative. This requires a different notion of interaction graph, and also the so-called ``corrected average density'' of the support list to be taken into account. The paper also present examples where these bounds are tight.
	%\mnote{Do we need more of a ``conclusion''/take away?}
	
%#########################################################
%#########################################################
%#########################################################
%#########################################################
\section{Selecting knapsack relaxation cuts}\label{sec:agg}
	%\mnoter{Better section title}
Recall that an important class of cuts via the method of structured relaxation (Section~\ref{sec:intro}), is obtained via the knapsack relaxation. There are numerous ways to generate a knapsack relaxation: one can just use the original constraints, but more generally, we can take a linear combination of the constraints and obtain a knapsack relaxation. We call such cuts as \emph{aggregation cuts.} Most cuts used by modern IP solvers (except perhaps clique cuts) can be described as aggregation cuts.

The set of all aggregation cuts have been studied empirically~\cite{FukasawaG11}, but not much is known from a theoretical perspective. An important question from the perspective of cutting-plane selection is: How should we construct these relaxations? What multipliers should we use? What happens if we select cuts from relaxations obtained only using individual constraints, without aggregating?

%While different methods have been developed to generate various families of cutting-planes, several of the most important families are obtained through the aggregation of the original constraints of the problem. These are special types of what we call \emph{aggregation cuts}, which are those generated as follows: given an IP formulation, we first obtain a single implied inequality by aggregating the original constraints, and then generate a cut valid for the integer hull of the set defined by this single inequality together with variable bounds.

%It is easy to see that \emph{Chv\'atal-Gomory (CG) cuts} are aggregation cuts: in fact, each CG cut is precisely the integer hull of the set defined by one aggregated inequality \emph{without} variable bounds. Aggregation cuts include many other classes of cuts, such as lifted knapsack covers inequalities~\cite{wolsey:1975,zemel:1978} and weight inequalities \cite{weismantel19970}. 

%	Given the ubiquity of aggregation cuts, one interesting question from the perspective of cut selection mentioned in the introduction is the following: \emph{For which types of MILPs are aggregations necessary/most important for obtaining good cuts?}

%#########################################################
%#########################################################

	\subsection{Aggregation cuts and packing and covering MILPs} \label{sec:agg1}
	
		The paper~\cite{Bodur2017} examines the strength of aggregation cuts for \emph{packing} and \emph{covering} MILPs. The main result is that for these classes of problems, even considering all infinitely many aggregations, offers limited help. More precisely, they show that the CG and more generally aggregation closures can be 2-approximated by simply generating the respective closures for each of the original constraints, without using any aggregations. Therefore, for these problems, in order to obtain cuts that are much stronger than original constraint cuts, one needs to consider more complicated cuts that cannot be generated through aggregations (e.g. clique cuts).
		
	To give an idea of what is behind these results, we focus here on a weaker version of the result for aggregation cuts and packing polyhedra. %More formally, a \emph{packing polyhedron} is of the form $Q = \{x \in \R^n_+ \mid Ax \le b \}$ where all the data $(A,b)$ is non-negative and rational. We are interested in the pure integer set $Q \cap \Z^n$. 	
We say that a packing polyhedron $\{x \in \R^n_+ : Ax \le b\}$ is \emph{pre-processed} if $0 \leq A_{ij} \le b_i$ for all $i,j$ (notice that otherwise all integer solutions have $x_j = 0$ so we can simply exclude this variable). 
	 Given a packing polyhedron $Q$, its \emph{aggregation closure} is defined as $$\A(Q) := \bigcap_{\lambda \in \R^m_+} \conv(\{x \in \Z^n_+ \mid \lambda^\top A x \le \lambda^\top b\}).$$ 
Given two packing sets\footnote{Packing \emph{sets} are of the form $\{x \in \R^n_+ \mid (a^i)^{\top} x \le b_i,\,\forall i \in I \}$ for a possibly infinite set $I$ where all $(a^i,b_i)$'s are non-negative, i.e., it is not necessarily polyhedral. This generalization is required because while the aggregation closure is a packing set, it is not \red{known} if it is polyhedral.} $U \supseteq V$, we say that $U$ is an \emph{$\alpha$-approximation} of $V$ if for all non-negative objective functions $c \in \R^n_+$ we have 
\begin{eqnarray}
\max\{c^\top x \mid x \in U\} \le \alpha \cdot \max\{c^\top x \mid x \in V\}.\label{eq:blowup}
\end{eqnarray}
Notice that since $U \supseteq V$, we have $\alpha \ge 1$.
	
	One of the results present in \cite{Bodur2017} is that the for pre-processed packing polyhedra the aggregation closure is quite close to the \emph{LP relaxation}.
	
	\begin{theorem} \label{thm:agg}
		For any pre-processed packing polyhedron $P$, $P$ itself is a 2-approximation of the aggregation closure $\mathcal{A}(P)$.
	\end{theorem}
	
	We remark that the LP relaxation $P$ can have a gap of factor up to $\Omega(n)$ to the integer hull (for example, in the standard formulation of the stable set problem), thus this result says that the aggregation closure can make very little progress.  	
	
	We now present the proof of Theorem \ref{thm:agg}. At a very high level it stems from the following simple principle, which may be useful in guiding the choice of relaxations to be used for generating cuts:
	
	\begin{quote}
		\emph{``Relaxations with {small integrality gap} may yield weak cuts."}%\mnote{If need to save space, can more this inline}
	\end{quote}
	
	This can be formalized in the case of packing sets as follows. Given a set $Q$, let $Q^I$ denote its integer hull. 
	
	\begin{lemma} \label{lemma:basicAgg}
		Let $P \in \R^n_+$ be packing set. Let $\mathcal{Q}$ be a family of packing sets containing $P$ (e.g., relaxations). Suppose there is $\alpha$ such that for all $Q \in \mathcal{Q}$, $Q$ is an $\alpha$-approximation of $Q^I$. Let $\tilde{P} = P \cap \bigcap_{Q \in \mathcal{Q}} Q^I$, i.e., the effect of the addition of all valid cuts for the sets $Q$. Then $P$ itself is an $\alpha$-approximation of $\tilde{P}$.\footnote{One can show that the integer hull of a packing set is also a packing set~\cite{Bodur2017}, and thus so is $\tilde{P}$.} 
	\end{lemma}
	
	\begin{proof}
		One can show that given any two packing sets $U \subseteq V$, $V$ is an $\alpha$-approximation of $U$ if and only if $V \subseteq \alpha U := \{\alpha x : x \in U\}$ (see \cite{Goe95} and the extension to the non-polyhedral case in \cite{Bodur2017}). Thus, by our assumptions we have $P \subseteq Q \subseteq \alpha Q^I$ for all $Q \in \mathcal{Q}$, and since $\alpha \ge 1$ we also have $P \subseteq \alpha P$. Therefore $$P \subseteq \alpha P \cap \bigcap_{Q \in \mathcal{Q}} \alpha Q^I = \alpha \left(\bigcap_{Q \in \mathcal{Q}} \alpha Q^I\right) = \alpha \tilde{P},$$ concluding the proof. 
	\end{proof}		

	Theorem \ref{thm:agg} is a quick consequence of this principle. 
	
	\begin{proof}[Proof of Theorem \ref{thm:agg}]
		For a fixed multiplier vector $\lambda$, let $P_{\lambda} = \{x \in \R^d_+ : \lambda^{\top} A x \le \lambda^{\top} b\}$ be the relaxation obtaining by the corresponding aggregation. Notice that $P_{\lambda}$ is also a pre-processed packing polyhedron that is actually a \emph{knapsack}. It is well-known and easy to show that for pre-processed knapsacks the LP relaxation is a 2-approximation of the IP hull~\cite{Bodur2017}, thus $P_{\lambda}$ is a 2-approximation of $(P_{\lambda})^I$. The aggregation closure $\mathcal{A}(P)$ is precisely the intersection of all the sets $(P_{\lambda})^I$, and so the result follows from the Lemma \ref{lemma:basicAgg}.
	\end{proof}
	
%	\mnoteless{It may be much, but could say ``Maybe this indicates that to get super useful cuts for packing/covering problems we need to go to $k$-aggregations?'' (whose integrality gap can be $\approx k$, I think).}

	The paper~\cite{Bodur2017} shows that results also hold for covering problem (with and without bounds) as well with suitable notions of ``pre-processed'' polyhedra. Moreover, the paper shows that all these pre-processings of arbitrary packing/covering polyhedra can be achieved by generating CG cuts for each of the rows individually, thus yielding a bound between the 1-row CG/aggregation closures and the full aggregation closure (see \cite{Bodur2017} for more detail).
%	\mnote{Any further conclusion?}
%#########################################################
%#########################################################

	\subsection{Aggregations for general MIPs}
	
%	\mnote{Change tone from ``we'' to ``the authors''? Also in other sections}
	The paper~\cite{Bodur2017} also shows that for general MILPs the aggregation closure can be much stronger than \red{the closure of the cuts implied by the original constraints of the MILPs individually.} More precisely, the \emph{1-row closure} $1\mathcal{A}(P)$ of a mixed-integer set $P = \{ x \in \R^d \times \Z^n : Ax \le b,\, x \ge 0\}$ is obtained by intersecting the integer hulls of all relaxations $\{ x \in \R^d \times \Z^n : (a_i)^\top x \le b_i,\, x \ge 0\}$  where \red{$a_i$ is the $i^{\textup{th}}$ row of $A$. They} show the following result. 
	
	\begin{theorem}[\cite{Bodur2017}]
		There is a family of mixed-integer sets where the 1-row closure is arbitrarily worse than the aggregation closure, namely for every $\alpha \ge 0$ there is a mixed-integer set $P$ and objective function $c$ such that 
		\begin{align*}
			\max\{ c^\top x : x \in 1\mathcal{A}(P)\} \ge \alpha \cdot \max\{ c^\top x : x \in \mathcal{A}(P)\}.
		\end{align*} 
	\end{theorem}	
	
	As always, such worst-case examples could be very artificial and this phenomenon \red{might not appear in practice.}
	% thus reducing the usefulness of aggregations. However, we contend that this may not be the case and present\mnote{tone} some computational experiments to support this claim.}
	Given the availability of reasonably robust CG cut separating procedure \cite{fischetti:lo:2007}, in our experiment we compare 1-row CG cuts \red{(CG cuts for each $\{ x \in \R^d \times \Z^n : (a_i)^\top x \le b_i,\, x \ge 0\}$)} versus general CG cuts. We study two classes of instances: random equality instances and the so-called market split instances~\cite{cornuejols:da:1999}. See details of instances in~\cite{Bodur2017}. It was empirically observed that the 1-row CG are significantly weaker than general CG cuts.

The message for cutting-plane selection is clear: For packing and covering IPs, it may be sufficient to generate cuts from knapsack cuts from original single constraints, while for more general IPs aggregations can be significantly useful. 
	
\subsection{Aggregations with sign patterns}\label{sec:sign}

We dig deeper into this question of packing and covering vs general IPs and the performance of aggregation cuts. Note that aggregation cuts can be generalized to multi-row aggregation: apply a $k$ different set of multipliers to the original constraints to obtain a $k$-constraint relaxation and add cut valid for this set.

Observe that for packing and covering IPs all the coefficients of all the variables in all the constraints have the same sign. Therefore when we aggregate constraints we are not able to ``cancel'' variables, i.e., the support of an aggregated constraint is exactly equal to the union of supports of the original constraints used for the aggregation. A natural conjecture for the fact that the aggregation (resp. multi-row aggregation, i.e.) closure is well approximated by the original (resp. multi-row) 1-row closure for packing and covering problems, is the fact that such cancellations do not occur for these problems. Indeed one of the key ideas used to obtain good candidate aggregations in the heuristic described in \cite{marchand:wo:1999} is to use aggregations that maximize the chances of a cancellation. 

In order to study the effect of cancellations, the paper~\cite{dey2017strength} examines the strength of aggregation closures vis-\`a-vis original row constraints for \emph{sign-pattern IPs}. A sign-pattern IP is a problem  of the form $\{x\in \Z^n_+:Ax\le b\}$ where a given variable has exactly the same sign in every constraint, i.e. for a given $j$, $A_{i,j}$ is either non-negative for all $i$ or non-positive for all $i$. Thus aggregations do not create cancellations. The results are interesting: On the one hand we are able to show that the aggregation closure for such sign-pattern IPs is {\it 2-approximated} by the original 1-row closure. On the other hand, the multi-row aggregation closure cannot be well approximated by the original multi-row closure. Therefore, these classes of integer programs show results that are in between packing and covering IPs on one side and general IPs on the other side.

%#########################################################
%#########################################################

%#########################################################
%#########################################################
%#########################################################
%#########################################################

	\section{Final remarks}
%	
%	
%	\blue{[Dump of possible stuff, not clear should add this.]
%	Some things are left out of this survey.
%	
%	The notion of ``extreme'' function in the group setting could help in selecting cuts, since they are the strongest \ldots. However, this notion does not really narrow down the good cuts: is was showed recently that every extreme function can be arbitrarily approximated by just a minimal function (Robert's paper, follows up of Amitabh-JHU postdoc).
%	
%	So one needs a more refined notion, for example the volume one introduced by Gomory-Johnson. Amitabh has a new paper on this.
%	
%	[Add Conforti-Wolsey]
%	}

%	\red{[TODO: Wrap things up a bit more]}

\red{In Section \ref{sec:intro}, we discussed various types of cutting-planes available to modern day solvers and the theoretical analysis of these cutting-planes that has been conducted. As we discussed, while the theoretical analysis conducted thus far is important, these results have so far failed to directly help in actual cutting-plane selection. In Section \ref{sec:motivate}, we then highlighted  the need to focus on alternative  questions in order to better develop the science of cutting-plane selection. We would like to throw a word of caution here. While the questions we have indicated may eventually be  important ones to pursue, \emph{the answers may not be unique.} Indeed, modern day solvers have become complex, highly integrated engineering marvels, where different components of the solvers are heavily inter-dependent, and therefore it may be altogether possible that one kind of cutting-plane selection improves the average solution time in the context of one code and miserably fails in conjunction with other codes! Moreover, different tasks performed by a solver need allocation of time and resource (memory) and the task of obtaining dual bounds is not necessarily the ``top priority'' for many industrial users of solvers. Given solvers must ``please all'', the choice of cutting-plane selection definitely must speak to different requirements and compulsions. Overall, all of this points to the fact that the questions asked in Section \ref{sec:intro} must be pursued rigorously, but specific answers for each solver may lie in balancing various competing goals in cutting-plane selection (as mentioned in Section~\ref{sec:motivate}). 
}

\red{Based on the above discussion, one is also tempted to ask if there is any value in attempting to pursue ``principled and scientific'' theory for cutting-plane selection at all. We believe that the answer is yes. Most solvers have implemented heuristics for cutting-planes selection that can be run cheaply and do not put too much pressure on other components of the solvers. To go back to the example of GMICs, one reason that most codes, before the work by~\cite{balas:ce:co:na:96},  perhaps separated one cut and resolved LP was that it was considered expensive to add too many cuts and increase the size of LPs significantly. As computer speeds improve and specific architectures (such as  parallel computing, GPUs etc.) steadily get better, it may be possible to revisit all the heuristics that current day solvers implement and improve them on the basis of hard science of integer programming and cutting-plane selection. Therefore, we believe a better understanding of cutting-plane selection in a systematic fashion is likely to improve solvers in the long run.}

We also hope that there are some specific take-aways related to cut-selection from the results  presented in Sections~\ref{sec:sparsity} and \ref{sec:agg}: For example, if we have a structured sparse packing IP, Section \ref{sec:agg} indicates that if we generate aggregation cuts, then we may restrict the selection of cuts from knapsack relaxations of the original constraints. The result of Section \ref{sec:sparseforsparse} presents indication of specific sparsity patterns of cuts to be selected, based on the sparsity pattern of the packing IP. We recognize that these results are very preliminary results; as such, we are very far from answering the questions raised in Section~\ref{sec:motivate}.
	
We end this paper with another important question: Is it possible to solve all (or any of) the questions raised in Section~\ref{sec:motivate} using \emph{machine learning?} Indeed, there is a significant trend of applying machine learning to integer programming. Recently this has been applied to branching~\cite{learnBranching,learnStrongBranching,Lodi2017}, node selection~\cite{JHUlearn},
%\mnote{is this the correct term?}, 
decompositions~\cite{lubbeckeLearning}, solver parameter selection~\cite{MLcpaior,MLijcai}, load balancing in parallel branch-and-bound~\cite{quentin}, probing in MINLPs~\cite{leeLearning}, and selecting solution methods for mixed-integer quadratic programs~\cite{lodiLearning}. To the best of our knowledge machine learning has not been applied to cut generation and selection. We believe that the answer to improved and principled cut-selection may be somewhere in between using only machine learning or only theoretical analysis. It is possible to envisage that results from theoretical analysis of cutting-plane selection work in concert with learning algorithms. Indeed, learning task models may be able to leverage specific theoretical knowledge about cutting planes selection.
%However, it does not subsume analytical work on better understanding cutting plane selection. In fact, to obtain significant improvements, it may be needed to use in the learning task models that leverage specific knowledge about cutting planes.

Overall, we hope that this review fosters inquiry into new questions and techniques that aid in cutting-plane selection, integrating practical knowledge and considerations with theoretical tools.
%, which are different from the traditional questions analyzing strength of cutting-planes.  
	
%#########################################################
%#########################################################

\section*{Acknowledgements}
We would like to thank Tobias Achterberg, Domenico Salvagnin and Roland Wunderling for their help with preparing this manuscript. \red{We would also like to thank anonymous reviewers for their feedback that greatly improved the presentation of this manuscript.} Santanu S. Dey would like to acknowledge the support of NSF CMMI grant \# 1562578. Marco Molinaro would like to acknowledge the support of CNPq grants Universal \#431480/2016-8 and Bolsa de Produtividade em Pesquisa \#310516/2017-0.

\bibliographystyle{plain}
\bibliography{Cuts}

\begin{thebibliography}{100}

\bibitem{tobiasAussoisTalk}
T.~Achterberg.
\newblock Exploiting degeneracy in mip.
\newblock Talk at Aussois 22nd Combinatorial Optimization Workshop, 2018,
  \url{http://www.iasi.cnr.it/aussois/web/uploads/2018/slides/achterbergt.pdf}.

\bibitem{tobiasMIPTalk}
T.~Achterberg.
\newblock {LP} basis selection and cutting planes.
\newblock Talk at MIP 2010,
  \url{http://www2.isye.gatech.edu/mip2010/program/program.pdf}.

\bibitem{tobiasThesis}
T.~Achterberg.
\newblock {\em Constraint Integer Programming}.
\newblock PhD thesis, Technische Universit\"at Berlin, 2007.

\bibitem{achterberg2009scip}
T.~Achterberg.
\newblock Scip: solving constraint integer programs.
\newblock {\em Math.~Program.~Comput.}, 1(1):1--41, 2009.

\bibitem{AchterbergR10}
T.~Achterberg and C.~Raack.
\newblock The {Mcf}-separator: detecting and exploiting multi-commodity flow
  structures in {MIPs}.
\newblock {\em Math. Program. Comput.}, 2(2):125--165, 2010.

\bibitem{agra:co:2007}
A.~Agra and M.~F. Constantino.
\newblock Lifting two-integer knapsack inequalities.
\newblock {\em Math.~Program.}, 109:115--154, 2007.

\bibitem{learnStrongBranching}
A.~M. Alvarez, Q.~Louveaux, and L.~Wehenkel.
\newblock A machine learning-based approximation of strong branching.
\newblock {\em INFORMS J.~Comput.}, 29(1):185--195, 2017.

\bibitem{quentin}
A.~M. Alvarez, L.~Wehenkel, and Q.~Louveaux.
\newblock Machine learning to balance the load in parallel branch-and-bound.
\newblock {\em http://hdl.handle.net/2268/181086}, 2015.

\bibitem{amaldi2014coordinated}
E.~Amaldi, S.~Coniglio, and S.~Gualandi.
\newblock Coordinated cutting plane generation via multi-objective separation.
\newblock {\em Math.~Program.}, 143(1-2):87--110, 2014.

\bibitem{andersen:co:li:2005}
K.~Andersen, G.~Cornu\'ejols, and Y.~Li.
\newblock Reduce-and-split cuts: {I}mproving the performance of mixed integer
  {G}omory cuts.
\newblock {\em Management~Sci.}, 51:1720--1732, 2005.

\bibitem{Andersen05}
K.~Andersen, G.~Cornu\'{e}jols, and Y.~Li.
\newblock Split closure and intersection cuts.
\newblock {\em Math.~Program.}, 102:457--493, 2005.

\bibitem{andersen:lo:we:wo:2007}
K.~Andersen, Q.~Louveaux, R.~Weismantel, and L.~A. Wolsey.
\newblock Inequalities from two rows of a simplex tableau.
\newblock In M.~Fischetti and D.~P. Williamson, editors, {\em Proceedings
  $12^{\textrm{th}}$ Conference on Integer Programming and Combinatorial
  Optimization}, pages 30--42. Springer-Verlag, 2007.

\bibitem{andreello2007embedding}
G.~Andreello, A.~Caprara, and M.~Fischetti.
\newblock Embedding $\{$0, $1/2$$\}$-cuts in a branch-and-cut framework: A
  computational study.
\newblock {\em INFORMS J.~Comput.}, 19(2):229--238, 2007.

\bibitem{araoz:ev:go:jo:2003}
J.~Ar\'aoz, L.~Evans, R.~E. Gomory, and E.~L. Johnson.
\newblock Cyclic groups and knapsack facets.
\newblock {\em Math.~Program.}, 96:377--408, 2003.

\bibitem{atamturk:2003}
A.~Atamt{\"u}rk.
\newblock On the facets of the mixed-integer knapsack polyhedron.
\newblock {\em Math.~Program.}, 98:145--175, 2003.

\bibitem{atamturk:2004}
A.~Atamt{\"u}rk.
\newblock Sequence independent lifting for mixed-integer programming.
\newblock {\em Oper.~Res.}, 52:487--490, 2004.

\bibitem{atamturk2010mingling}
A.~Atamt{\"u}rk and O.~G{\"u}nl{\"u}k.
\newblock Mingling: mixed-integer rounding with bounds.
\newblock {\em Math.~Program.}, 123(2):315--338, 2010.

\bibitem{atamturk2017path}
A.~Atamt\"{u}rk, S.~K\"{u}\c{c}\"{u}kyavuz, and B.~Tezel.
\newblock Path cover and path pack inequalities for the capacitated
  fixed-charge network flow problem.
\newblock {\em SIAM J.~Optim.}, 27(3):1943--1976, 2017.

\bibitem{atamturk:ne:sa:2000}
A.~Atamt{\"u}rk, G.~L. Nemhauser, and M.~W.~P. Savelsbergh.
\newblock Conflict graphs in integer programming.
\newblock {\em European~J.~Oper.~Res.}, 121:40--55, 2000.

\bibitem{au2016comprehensive}
Y.~H. Au and L.~Tun{\c{c}}el.
\newblock A comprehensive analysis of polyhedral lift-and-project methods.
\newblock {\em SIAM J.~Discrete Math.}, 30(1):411--451, 2016.

\bibitem{au2016elementary}
Y.~H. Au and L.~Tun{\c{c}}el.
\newblock Elementary polytopes with high lift-and-project ranks for strong
  positive semidefinite operators.
\newblock {\em arXiv preprint arXiv:1608.07647}, 2016.

\bibitem{averkov2015lifting}
G.~Averkov and A.~Basu.
\newblock Lifting properties of maximal lattice-free polyhedra.
\newblock {\em Math.~Program.}, 154(1-2):81--111, 2015.

\bibitem{averkov2017approximation}
G.~Averkov, A.~Basu, and J.~Paat.
\newblock Approximation of corner polyhedra with families of intersection cuts.
\newblock In {\em International Conference on Integer Programming and
  Combinatorial Optimization}, pages 51--62. Springer, 2017.

\bibitem{AwateCGT15}
Y.~Awate, G.~Cornu{\'{e}}jols, B.~Guenin, and L.~Tun{\c{c}}el.
\newblock On the relative strength of families of intersection cuts arising
  from pairs of tableau constraints in mixed integer programs.
\newblock {\em Math.~Program.}, 150(2):459--489, 2015.

\bibitem{balas:1971}
E.~Balas.
\newblock Intersection cuts - a new type of cutting planes for integer
  programming.
\newblock {\em Oper.~Res.}, 19:19--39, 1971.

\bibitem{balas:1975}
E.~Balas.
\newblock Facets of the knapsack polytope.
\newblock {\em Math.~Program.}, 8:146--164, 1975.

\bibitem{balas:1979}
E.~Balas.
\newblock Disjunctive programming.
\newblock {\em Ann.~Disc.~Math.}, 5:3--51, 1979.

\bibitem{balas2007new}
E.~Balas and P.~Bonami.
\newblock New variants of lift-and-project cut generation from the lp tableau:
  open source implementation and testing.
\newblock In {\em International Conference on Integer Programming and
  Combinatorial Optimization}, pages 89--103. Springer, 2007.

\bibitem{balas:ce:co:93}
E.~Balas, S.~Ceria, and G.~{Cornu\'ejols}.
\newblock A lift-and-project cutting plane algorithm for mixed integer 0-1
  programs.
\newblock {\em Math.~Program.}, 58:295--324, 1993.

\bibitem{balas:ce:co:na:96}
E.~Balas, S.~Ceria, G.~{Cornu\'ejols}, and N.~Natraj.
\newblock Gomory cuts revisited.
\newblock {\em Oper.~Res.~Lett.}, 19:1--9, 1996.

\bibitem{BCC96}
E.~Balas, S.~Ceria, and G.~Cornuéjols.
\newblock Mixed 0-1 programming by lift-and-project in a branch-and-cut
  framework.
\newblock {\em Management~Sci.}, 42(9):1229--1246, 1996.

\bibitem{balas:je:1980}
E.~Balas and R.~Jeroslow.
\newblock Strenghtening cuts for mixed integer programs.
\newblock {\em European~J.~Oper.~Res.}, 4:224--234, 1980.

\bibitem{Balas2013}
E.~Balas and F.~Margot.
\newblock Generalized intersection cuts and a new cut generating paradigm.
\newblock {\em Math.~Program.}, 137(1):19--35, Feb 2013.

\bibitem{balas:pe:2002}
E.~Balas and M.~Perregaard.
\newblock Lift-and-project for mixed 0-1 programming: recent progress.
\newblock {\em Discr.~Appl.~Math.}, 123:129--154, 2002.

\bibitem{balas:ze:1984}
E.~Balas and E.~Zemel.
\newblock Facets of knapsack polytope from minimal covers.
\newblock {\em SIAM J.~Appl.~Math.}, 34:119--148, 1984.

\bibitem{basu:bo:co:ma:2010}
A.~Basu, P.~Bonami, G.~Cornu{\'e}jols, and F.~Margot.
\newblock Experiments with two-row cuts from degenerate tableaux.
\newblock {\em INFORMS J.~Comput.}, 23(4):578--590, 2011.

\bibitem{basu:bo:co:ma:2009}
A.~Basu, P.~Bonami, G.~Cornu{\'e}jols, and F.~Margot.
\newblock On the relative strength of split, triangle and qudrilateral cuts.
\newblock {\em Math.~Program.}, 126:281--314, 2011.

\bibitem{basu:co:co:za:2009a}
A.~Basu, M.~Conforti, G.~Cornu{\'e}jols, and G.~Zambelli.
\newblock Maximal lattice-free convex sets in linear subspaces.
\newblock {\em Math.~Oper.~Res.}, 35:704--720, 2010.

\bibitem{basu2015geometric}
A.~Basu, M.~Conforti, and M.~Di~Summa.
\newblock A geometric approach to cut-generating functions.
\newblock {\em Math.~Program.}, 151(1):153--189, 2015.

\bibitem{basu2017optimal}
A.~Basu, M.~Conforti, and M.~Di~Summa.
\newblock Optimal cutting planes from the group relaxations.
\newblock {\em arXiv preprint arXiv:1710.07672}, 2017.

\bibitem{basu:co:ko:2011}
A.~Basu, G.~Cornu{\'e}jols, and M.~K{\"o}ppe.
\newblock Unique minimal liftings for simplicial polytopes.
\newblock {\em Math.~Oper.~Res.}, 37(2):346--355, 2012.

\bibitem{BasuCM12}
A.~Basu, G.~Cornu{\'{e}}jols, and F.~Margot.
\newblock Intersection cuts with infinite split rank.
\newblock {\em Math. Oper. Res.}, 37(1):21--40, 2012.

\bibitem{basu2011probabilistic}
A.~Basu, G.~Cornu{\'e}jols, and M.~Molinaro.
\newblock A probabilistic analysis of the strength of the split and triangle
  closures.
\newblock In {\em IPCO}, pages 27--38. Springer, 2011.

\bibitem{basu2014equivariant}
A.~Basu, R.~Hildebrand, and M.~K{\"o}ppe.
\newblock Equivariant perturbation in gomory and johnson's infinite group
  problem. i. the one-dimensional case.
\newblock {\em Math.~Oper.~Res.}, 40(1):105--129, 2014.

\bibitem{basu2016light1}
A.~Basu, R.~Hildebrand, and M.~K{\"o}ppe.
\newblock Light on the infinite group relaxation i: foundations and taxonomy.
\newblock {\em 4OR}, 14(1):1--40, 2016.

\bibitem{basu2016light2}
A.~Basu, R.~Hildebrand, and M.~K{\"o}ppe.
\newblock Light on the infinite group relaxation ii: sufficient conditions for
  extremality, sequences, and algorithms.
\newblock {\em 4OR}, 14(2):107--131, 2016.

\bibitem{basu2013k+1}
A.~Basu, R.~Hildebrand, M.~Koppe, and M.~Molinaro.
\newblock A (k+1)-slope theorem for the k-dimensional infinite group
  relaxation.
\newblock {\em SIAM J.~Optim.}, 23(2):1021--1040, 2013.

\bibitem{basu2016minimal}
A.~Basu, R.~Hildebrand, and M.~Molinaro.
\newblock Minimal cut-generating functions are nearly extreme.
\newblock In {\em International Conference on Integer Programming and
  Combinatorial Optimization}, pages 202--213. Springer, 2016.

\bibitem{basu2015operations}
A.~Basu and J.~Paat.
\newblock Operations that preserve the covering property of the lifting region.
\newblock {\em SIAM J.~Optim.}, 25(4):2313--2333, 2015.

\bibitem{BasuEqui2017}
Amitabh Basu, Robert Hildebrand, and Matthias K{\"o}ppe.
\newblock Equivariant perturbation in gomory and johnson's infinite group
  problem---iii: foundations for the k-dimensional case with applications to
  {k=2}.
\newblock {\em Mathematical Programming}, 163(1):301--358, May 2017.

\bibitem{bell:sh:1977}
D.~E. Bell and J.~F. Shapiro.
\newblock A convergent duality theory for integer programming.
\newblock {\em OR}, 25:419--434, 1977.

\bibitem{benchetrit2018characterizing}
Yohann Benchetrit, Samuel Fiorini, Tony Huynh, and Stefan Weltge.
\newblock Characterizing polytopes in the 0/1-cube with bounded
  chv{\'a}tal-gomory rank.
\newblock {\em Mathematics of Operations Research}, 2018.

\bibitem{bienstock2004subset}
D.~Bienstock and M.~Zuckerberg.
\newblock Subset algebra lift operators for 0-1 integer programming.
\newblock {\em SIAM J.~Optim.}, 15(1):63--95, 2004.

\bibitem{bodur2017cutting}
M.~Bodur, S.~Dash, and O.~G{\"u}nl{\"u}k.
\newblock Cutting planes from extended {LP} formulations.
\newblock {\em Math.~Program.}, 161(1-2):159--192, 2017.

\bibitem{Boduretal}
M.~Bodur, S.~Dash, O.~Günlük, and J.~Luedtke.
\newblock Strengthened benders cuts for stochastic integer programs with
  continuous recourse.
\newblock {\em INFORMS J.~Comput.}, 29(1):77--91, 2017.

\bibitem{Bodur2017}
M.~Bodur, A.~Del~Pia, S.~S. Dey, M.~Molinaro, and S.~Pokutta.
\newblock Aggregation-based cutting-planes for packing and covering integer
  programs.
\newblock {\em Math.~Program.}, Sep 2017.

\bibitem{bodur2017lower}
Merve Bodur, Alberto Del~Pia, Santanu~S Dey, and Marco Molinaro.
\newblock Lower bounds on the lattice-free rank for packing and covering
  integer programs.
\newblock {\em arXiv preprint arXiv:1710.00031}, 2017.

\bibitem{Bonami2012}
P.~Bonami.
\newblock On optimizing over lift-and-project closures.
\newblock {\em Math.~Program.~Comput.}, 4(2):151--179, Jun 2012.

\bibitem{lodiLearning}
P.~Bonami, A.~Lodi, and G.~Zarpellon.
\newblock Learning a classification of mixed-integer quadratic programming
  problems.
\newblock
  \url{http://cerc-datascience.polymtl.ca/wp-content/uploads/2018/01/Technical-Report_DS4DM-2017-013.pdf},
  2017.

\bibitem{bonami2014cut}
P.~Bonami and F.~Margot.
\newblock Cut generation through binarization.
\newblock In {\em International Conference on Integer Programming and
  Combinatorial Optimization}, pages 174--185. Springer, 2014.

\bibitem{bonami2005using}
P.~Bonami and M.~Minoux.
\newblock Using rank-1 lift-and-project closures to generate cuts for 0--1
  mips, a computational investigation.
\newblock {\em Discrete Optim.}, 2(4):288--307, 2005.

\bibitem{borozan:2007}
V.~Borozan and G.~{Cornu\'ejols}.
\newblock Minimal valid inequalities for integer constraints.
\newblock {\em Math.~Oper.~Res.}, 34:538--546, 2009.

\bibitem{concentration}
S.~Boucheron, G.~Lugosi, and P.~Massart.
\newblock {\em Concentration Inequalities: A Nonasymptotic Theory of
  Independence}.
\newblock OUP Oxford, 2013.

\bibitem{Cadoux2010}
Florent Cadoux.
\newblock Computing deep facet-defining disjunctive cuts for mixed-integer
  programming.
\newblock {\em Mathematical Programming}, 122(2):197--223, Apr 2010.

\bibitem{caprara:fi:1996}
A.~Caprara and M.~Fischetti.
\newblock $\{ 0,\frac{1}{2} \}$- {Chv\'atal-Gomory} cuts.
\newblock {\em Math.~Program.}, 74:221--235, 1996.

\bibitem{Caroephd}
C.~C. Car{\o}e.
\newblock {\em Decomposition in Stochastic Integer Programming}.
\newblock PhD thesis, Institute of Mathematical Sciences, Department of
  Operations Research, University of Copenhagen, Denmark, 1998.

\bibitem{carr2000strengthening}
R.~D. Carr, L.~Fleischer, V.~J. Leung, and C.~A Phillips.
\newblock Strengthening integrality gaps for capacitated network design and
  covering problems.
\newblock In {\em SODA}, pages 106--115, 2000.

\bibitem{ChandrasekaranV16}
K.~Chandrasekaran, L.~A. V{\'{e}}gh, and S.~S. Vempala.
\newblock The cutting plane method is polynomial for perfect matchings.
\newblock {\em Math.~Oper.~Res.}, 41(1):23--48, 2016.

\bibitem{ChenKS11}
B.~Chen, S.~K{\"{u}}{\c{c}}{\"{u}}kyavuz, and S.~Sen.
\newblock Finite disjunctive programming characterizations for general
  mixed-integer linear programs.
\newblock {\em Oper.~Res.}, 59(1):202--210, 2011.

\bibitem{chen2012computational}
B.~Chen, S.~K{\"u}{\c{c}}{\"u}kyavuz, and S.~Sen.
\newblock A computational study of the cutting plane tree algorithm for general
  mixed-integer linear programs.
\newblock {\em Oper.~Res.~Lett.}, 40(1):15--19, 2012.

\bibitem{cheung2007computation}
K.~KH. Cheung.
\newblock Computation of the lasserre ranks of some polytopes.
\newblock {\em Mathematics of Operations Research}, 32(1):88--94, 2007.

\bibitem{ChvatalCH89}
V.~Chv\'atal, W.~Cook, and M.~Hartmann.
\newblock On cutting-plane proofs in combinatorial optimization.
\newblock {\em Lin.~Alg.~Appl.}, {114/115}:{455--499}, {1989}.

\bibitem{chvatal1973}
V.~Chvátal.
\newblock Edmonds polytopes and a hierarchy of combinatorial problems.
\newblock {\em Discrete Mathematics}, 4(4):305 -- 337, 1973.

\bibitem{Coleman84}
T.~F. Coleman.
\newblock {\em Large Sparse Numerical Optimization}, volume 165 of {\em Lecture
  Notes in Computer Science}.
\newblock Springer, 1984.

\bibitem{conforti2014cut}
M.~Conforti, G.~Cornu{\'e}jols, A.~Daniilidis, C.~Lemar{\'e}chal, and
  J.~Malick.
\newblock Cut-generating functions and s-free sets.
\newblock {\em Math.~Oper.~Res.}, 40(2):276--391, 2014.

\bibitem{co:co:za:2009b}
M.~Conforti, G.~Cornu{\'e}jols, and G.~Zambelli.
\newblock A geometric perspective on lifting.
\newblock {\em Oper.~Res.}, 59:569--577, 2009.

\bibitem{IPCCZbook}
M.~Conforti, G.~Cornu{\'e}jols, and G.~Zambelli.
\newblock {\em Integer Programming}.
\newblock Springer, 2012.

\bibitem{ConfortiPSFSpli15}
M.~Conforti, A.~Del Pia, M.~Di Summa, and Y.~Faenza.
\newblock Reverse split rank.
\newblock {\em Math. Program.}, 154(1-2):273--303, 2015.

\bibitem{ConfortiPSFGCG15}
M.~Conforti, A.~Del Pia, M.~Di Summa, Y.~Faenza, and R.~Grappe.
\newblock Reverse {C}hv{\'{a}}tal-{G}omory rank.
\newblock {\em {SIAM} J. Discrete Math.}, 29(1):166--181, 2015.

\bibitem{conforti2016facet}
M.~Conforti and L.~A Wolsey.
\newblock “facet” separation with one linear program.
\newblock Technical report, Universit{\'e} catholique de Louvain, Center for
  Operations Research and Econometrics (CORE), 2016.

\bibitem{coniglio}
Stefano Coniglio and Martin Tieves.
\newblock On the generation of cutting planes which maximize the bound
  improvement.
\newblock In Evripidis Bampis, editor, {\em Experimental Algorithms}, pages
  97--109, Cham, 2015. Springer International Publishing.

\bibitem{cook1990cutting}
W.~Cook.
\newblock Cutting-plane proofs in polynomial space.
\newblock {\em Math.~Program.}, 47(1-3):11--18, 1990.

\bibitem{cook2001matrix}
W.~Cook and S.~Dash.
\newblock On the matrix-cut rank of polyhedra.
\newblock {\em Math.~Oper.~Res.}, 26(1):19--30, 2001.

\bibitem{cookSafe}
W.~Cook, S.~Dash, R.~Fukasawa, and M.~Goycoolea.
\newblock Numerically safe gomory mixed-integer cuts.
\newblock {\em INFORMS J.~Comput.}, 21(4):641--649, 2009.

\bibitem{cook:ka:sc:1990}
W.~Cook, R.~Kannan, and A.~Schrijver.
\newblock Ch{v\'a}tal closures for mixed integer programming problems.
\newblock {\em Math.~Program.}, 58:155--174, 1990.

\bibitem{Cook2013}
W.~Cook, T.~Koch, D.~E. Steffy, and K.~Wolter.
\newblock A hybrid branch-and-bound approach for exact rational mixed-integer
  programming.
\newblock {\em Math.~Program.~Comput.}, 5(3):305--344, Sep 2013.

\bibitem{cornuejols:2007}
G.~Cor{nu\'ej}ols.
\newblock Revival of the {G}omory cuts in the 1990's.
\newblock {\em Ann.~Oper.~Res.}, 149:63--66, 2007.

\bibitem{cornuejols:da:1999}
G.~Corn{u\'e}jols and M.~Dawande.
\newblock A class of hard small 0-1 programs.
\newblock {\em INFORMS J.~Comput.}, 11:205--210, 1999.

\bibitem{CornuejolsLemarechal}
G.~Cornu{\'e}jols and C.~Lemar{\'e}chal.
\newblock A convex-analysis perspective on disjunctive cuts.
\newblock {\em Mathematical Programming}, 106(3):567--586, May 2006.

\bibitem{CornuejolsLi01}
G.~Corn{\'u}ejols and Y.~Li.
\newblock {Elementary Closures for Integer Programs}.
\newblock {\em Oper.~Res.~Lett.}, {28}:{1--8}, {2001}.

\bibitem{cornuejols:li:2002}
G.~Corn{u\'e}jols and Y.~Li.
\newblock On the rank of mixed 0-1 polyhedra.
\newblock {\em Math.~Program.}, 91:391--397, 2002.

\bibitem{cornuejols:li:va:2003}
G.~Cor{nu\'ej}ols, Y.~Li, and D.~Vandenbussche.
\newblock K-cuts: a variation of gomory mixed integer cuts from the {LP}
  tableau.
\newblock {\em INFORMS J.~Comput.}, 15:385--396, 2003.

\bibitem{margotGerardSafe}
G.~Cornu{\'e}jols, F.~Margot, and G.~Nannicini.
\newblock On the safety of gomory cut generators.
\newblock {\em Math.~Program.~Comput.}, 5(4):345--395, Dec 2013.

\bibitem{cornuejols2012tight}
G.~Cornu{\'e}jols, C.~Michini, and G.~Nannicini.
\newblock How tight is the corner relaxation? {Insights} gained from the stable
  set problem.
\newblock {\em Discrete Optim.}, 9(2):109--121, 2012.

\bibitem{cornuejols2016some}
G{\'e}rard Cornu{\'e}jols and Dabeen Lee.
\newblock On some polytopes contained in the 0, 1 hypercube that have a small
  chv{\'a}tal rank.
\newblock In {\em International Conference on Integer Programming and
  Combinatorial Optimization}, pages 300--311. Springer, 2016.

\bibitem{Cornuejols2013}
G{\'e}rard Cornu{\'e}jols and Marco Molinaro.
\newblock A 3-slope theorem for the infinite relaxation in the plane.
\newblock {\em Mathematical Programming}, 142(1):83--105, Dec 2013.

\bibitem{crowder:jo:pa:1983}
H.~Crowder, E.~L. Johnson, and M.~W. Padberg.
\newblock Solving large scale zero-one linear programming problem.
\newblock {\em Oper.~Res.}, 31:803--834, 1983.

\bibitem{dash:2005}
S.~Dash.
\newblock Exponential lower bounds on the lengths of some classes of
  branch-and-cut proofs.
\newblock {\em Naval~Res.~Logist.}, 30:678--700, 2005.

\bibitem{DashDG12}
S.~Dash, S.~S. Dey, and O.~G{\"{u}}nl{\"{u}}k.
\newblock Two dimensional lattice-free cuts and asymmetric disjunctions for
  mixed-integer polyhedra.
\newblock {\em Math.~Program.}, 135(1-2):221--254, 2012.

\bibitem{dash2010heuristic}
S.~Dash and M.~Goycoolea.
\newblock A heuristic to generate rank-1 {GMI} cuts.
\newblock {\em Math.~Program.~Comput.}, 2(3-4):231--257, 2010.

\bibitem{dash2010two}
S.~Dash, M.~Goycoolea, and O.~G{\"u}nl{\"u}k.
\newblock Two-step {MIR} inequalities for mixed integer programs.
\newblock {\em INFORMS J.~Comput.}, 22(2):236--249, 2010.

\bibitem{dash2006valid}
S.~Dash and O.~G{\"u}nl{\"u}k.
\newblock Valid inequalities based on simple mixed-integer sets.
\newblock {\em Math.~Program.}, 105(1):29--53, 2006.

\bibitem{dash2009mixing}
S.~Dash and O.~G{\"u}nl{\"u}k.
\newblock On mixing inequalities: Rank, closure, and cutting-plane proofs.
\newblock {\em SIAM J.~Optim.}, 20(2):1090--1109, 2009.

\bibitem{DashG13}
S.~Dash and O.~G{\"{u}}nl{\"{u}}k.
\newblock On t-branch split cuts for mixed-integer programs.
\newblock {\em Math.~Program.}, 141(1-2):591--599, 2013.

\bibitem{dash2015relative}
S.~Dash, O.~G{\"u}nl{\"u}k, and M.~Molinaro.
\newblock On the relative strength of different generalizations of split cuts.
\newblock {\em Discrete Optim.}, 16:36--50, 2015.

\bibitem{dash2014computational}
S.~Dash, O.~G{\"u}nl{\"u}k, and J.~P. Vielma.
\newblock Computational experiments with cross and crooked cross cuts.
\newblock {\em INFORMS J.~Comput.}, 26(4):780--797, 2014.

\bibitem{nato}
Sanjeeb Dash.
\newblock Mixed integer rounding cuts and master group polyhedra.
\newblock {\em Combinatorial Optimization: Methods and Applications}, 31:1--32,
  2011.

\bibitem{DASH2011305}
Sanjeeb Dash, Santanu~S. Dey, and Oktay Günlük.
\newblock On mixed-integer sets with two integer variables.
\newblock {\em Operations Research Letters}, 39(5):305 -- 309, 2011.

\bibitem{del2012convergence}
A.~Del~Pia and R.~Weismantel.
\newblock On convergence in mixed integer programming.
\newblock {\em Math.~Program.}, pages 1--16, 2012.

\bibitem{del2012rank}
Alberto Del~Pia.
\newblock On the rank of disjunctive cuts.
\newblock {\em Mathematics of Operations Research}, 37(2):372--378, 2012.

\bibitem{dey:lowerbnd:2009}
S.~S. Dey.
\newblock A note on the split rank of intersection cuts.
\newblock {\em Math.~Program.}, 130(1):107--124, 2011.

\bibitem{dey2015some}
S.~S. Dey, A.~Iroume, and M.~Molinaro.
\newblock Some lower bounds on sparse outer approximations of polytopes.
\newblock {\em Oper.~Res.~Lett.}, 43(3):323--328, 2015.

\bibitem{dey2017strength}
S.~S. Dey, A.~Iroume, and G.~Wang.
\newblock The strength of multi-row aggregation cuts for sign-pattern integer
  programs.
\newblock {\em arXiv preprint arXiv:1711.06963}, 2017.

\bibitem{dey:lo:wo:tr:2010}
S.~S. Dey, A.~Lodi, L.~A. Wolsey, and A.~Tramontani.
\newblock Experiments with two row tableau cuts.
\newblock In F.~Eisenbrand and B.~Shepherd, editors, {\em Proceedings
  $14^{\textrm{th}}$ Conference on Integer Programming and Combinatorial
  Optimization}, pages 424--437. Springer-Verlag, 2010.

\bibitem{DeyL11}
S.~S. Dey and Q.~Louveaux.
\newblock Split rank of triangle and quadrilateral inequalities.
\newblock {\em Math.~Oper.~Res.}, 36(3):432--461, 2011.

\bibitem{deyMolinaroWang:2017}
S.~S. Dey, M.~Molinaro, and Q.~Wang.
\newblock Analysis of sparse cutting planes for sparse milps with applications
  to stochastic milps.
\newblock {\em Math.~Oper.~Res.}, 0(0):null, 0.

\bibitem{deyMolinaroWang:2015}
S.~S. Dey, M.~Molinaro, and Q.~Wang.
\newblock Approximating polyhedra with sparse inequalities.
\newblock {\em Math.~Program.}, pages 1--24, 2015.

\bibitem{dey:ri:2008}
S.~S. Dey and J.-P.~P. Richard.
\newblock Facets of the two-dimensional infinite group problems.
\newblock {\em Math.~Oper.~Res.}, 33:140--166, 2008.

\bibitem{dey2010relations}
S.~S Dey and J.-P.~P Richard.
\newblock Relations between facets of low-and high-dimensional group problems.
\newblock {\em Math.~Program.}, 123(2):285--313, 2010.

\bibitem{DBLP:journals/mp/DeyRLM10}
S.~S. Dey, J.-P.~P. Richard, Y.~Li, and L.~A. Miller.
\newblock On the extreme inequalities of infinite group problems.
\newblock {\em Math.~Program.}, 121:145--170, 2010.

\bibitem{dey:wo:mixing:2010}
S.~S. Dey and L.~A. Wolsey.
\newblock Composite lifting of group inequalities and an application to two-row
  mixing inequalities.
\newblock {\em Discrete Optim.}, 7:256--268, 2010.

\bibitem{deywolseysfree}
S.~S. Dey and L.~A. Wolsey.
\newblock Constrained infinite group relaxations of {MIP}s.
\newblock {\em SIAM J.~Optim.}, 20:2890--2912, 2010.

\bibitem{dey:wo:2008a}
S.~S. Dey and L.~A. Wolsey.
\newblock Two row mixed-integer cuts via lifting.
\newblock {\em Math.~Program.}, 124(1-2):143--174, 2010.

\bibitem{espinoza:2008}
D.~Espinoza.
\newblock Computing with multiple-row {G}omory cuts.
\newblock In A.~Lodi, A.~Panconesi, and G.~Rinaldi, editors, {\em Proceedings
  $13^{\textrm{th}}$ Conference on Integer Programming and Combinatorial
  Optimization}, pages 214--224. Springer-Verlag, 2008.

\bibitem{fischetti:lo:2007}
M.~Fischetti and A.~Lodi.
\newblock Optimizing over the first {C}hv{\'a}tal closure.
\newblock {\em Math.~Program.}, 110:3--20, 2007.

\bibitem{fischetti2011relax}
M.~Fischetti and D.~Salvagnin.
\newblock A relax-and-cut framework for gomory mixed-integer cuts.
\newblock {\em Math.~Program.~Comput.}, 3(2):79--102, 2011.

\bibitem{FukasawaG11}
Ricardo Fukasawa and Marcos Goycoolea.
\newblock On the exact separation of mixed integer knapsack cuts.
\newblock {\em Math.~Program.}, 128(1-2):19--41, 2011.

\bibitem{GentileVW06}
C.~Gentile, P.~Ventura, and R.~Weismantel.
\newblock Mod-2 cuts generation yields the convex hull of bounded integer
  feasible sets.
\newblock {\em SIAM J.~Discrete Math.}, 20(4):913--919, 2006.

\bibitem{SCIP5}
A.~Gleixner, L.~Eifler, T.~Gally, G.~Gamrath, P.~Gemander, R.~L. Gottwald,
  G.~Hendel, C.~Hojny, T.~Koch, M.~Miltenberger, B.~M{\"u}ller, M.~E. Pfetsch,
  C.~Puchert, D.~Rehfeldt, F.~Schl{\"o}sser, F.~Serrano, Y.~Shinano, J.~Me.
  Viernickel, S.~Vigerske, D.~Weninger, J.~T. Witt, and J.~Witzig.
\newblock The scip optimization suite 5.0.
\newblock Technical Report 17-61, ZIB, Takustr.7, 14195 Berlin, 2017.

\bibitem{Goe95}
M.~X. Goemans.
\newblock Worst-case comparison of valid inequalities for the tsp.
\newblock {\em Math.~Program.}, 69:335--349, 1995.

\bibitem{Gomory58}
R.~E. Gomory.
\newblock Outline of an algorithm for integer solutions to linear programs.
\newblock {\em Bulletin of the American Mathematical Society}, 64:275--278,
  1958.

\bibitem{gomory:1960b}
R.~E. Gomory.
\newblock An algorithm for the mixed integer problem.
\newblock Technical Report RM-2597, RAND Corporation, 1960.

\bibitem{gomory:1969}
R.~E. Gomory.
\newblock Some polyhedra related to combinatorial problems.
\newblock {\em Lin.~Alg.~Appl.}, 2:341--375, 1969.

\bibitem{gomory:jo:1972a}
R.~E. Gomory and E.~L. Johnson.
\newblock Some continuous functions related to corner polyhedra, part {I}.
\newblock {\em Math.~Program.}, 3:23--85, 1972.

\bibitem{gomory:jo:1972b}
R.~E. Gomory and E.~L. Johnson.
\newblock Some continuous functions related to corner polyhedra, part {II}.
\newblock {\em Math.~Program.}, 3:359--389, 1972.

\bibitem{gomory:jo:2003}
R.~E. Gomory and E.~L. Johnson.
\newblock T-space and cutting planes.
\newblock {\em Math.~Program.}, 96:341--375, 2003.

\bibitem{gomory:jo:ev:2003}
R.~E. Gomory, E.~L. Johnson, and L.~Evans.
\newblock Corner polyhedra and their connection with cutting planes.
\newblock {\em Math.~Program.}, 96:321--339, 2003.

\bibitem{gu:ne:sa:1999}
Z.~Gu, G.~L. Nemhauser, and M.~W.~P. Savelsbergh.
\newblock Lifted flow cover inequalities for mixed 0-1 integer programs.
\newblock {\em Math.~Program.}, 85:439--467, 1999.

\bibitem{pggthesis}
P.~Guerrero-Gar{\'c}ia.
\newblock {\em Range-Space Methods for Sparse Linear Programs}.
\newblock PhD thesis, Department of Applied Mathematics, University of Malaga,
  Spain, Spain, 2002.

\bibitem{gunluk:po:2001}
O.~G{\"u}nl{\"u}k and Y.~Pochet.
\newblock Mixing mixed-integer inequalities.
\newblock {\em Math.~Program.}, 90:429--457, 2001.

\bibitem{guzelsoy2007duality}
Menal Guzelsoy and Theodore~K Ralphs.
\newblock Duality for mixed-integer linear programs.
\newblock {\em International Journal of Operations Research}, 4(3):118--137,
  2007.

\bibitem{hammer1975facet}
P.~L Hammer, E.~L Johnson, and U.~N. Peled.
\newblock Facet of regular 0--1 polytopes.
\newblock {\em Math.~Program.}, 8(1):179--206, 1975.

\bibitem{JHUlearn}
H.~He, H.~Daum{\'e}, III, and J.~Eisner.
\newblock Learning to search in branch-and-bound algorithms.
\newblock In {\em Proceedings of the 27th International Conference on Neural
  Information Processing Systems - Volume 2}, NIPS'14, pages 3293--3301,
  Cambridge, MA, USA, 2014. MIT Press.

\bibitem{he2011probabilistic}
Q.~He, S.~Ahmed, and G.~L Nemhauser.
\newblock A probabilistic comparison of split and type 1 triangle cuts for
  two-row mixed-integer programs.
\newblock {\em SIAM J.~Optim.}, 21(3):617--632, 2011.

\bibitem{hoffman:pa:1991}
K.~L. Hoffman and M.~Padberg.
\newblock Improving {LP}-representation of zero-one linear programs for
  branch-and-cut.
\newblock {\em ORSA Journal of Computing}, 3:121--134, 1991.

\bibitem{MLcpaior}
F.~Hutter, H.~H. Hoos, and K.~Leyton{-}Brown.
\newblock Automated configuration of mixed integer programming solvers.
\newblock In {\em Integration of {AI} and {OR} Techniques in Constraint
  Programming for Combinatorial Optimization Problems, 7th International
  Conference, {CPAIOR} 2010, Bologna, Italy, June 14-18, 2010. Proceedings},
  pages 186--202, 2010.

\bibitem{jeroslow1978cutting}
R~Jeroslow.
\newblock Cutting-plane theory: Algebraic methods.
\newblock {\em Discrete mathematics}, 23(2):121--150, 1978.

\bibitem{jeroslow1979minimal}
Robert~G Jeroslow.
\newblock Minimal inequalities.
\newblock {\em Mathematical Programming}, 17(1):1--15, 1979.

\bibitem{johnson:1974}
E.~L. Johnson.
\newblock On the group problem for mixed integer programming.
\newblock {\em Math.~Program.~Study}, 2:137--179, 1974.

\bibitem{johnson:1979}
E.~L. Johnson.
\newblock On the group problem and a subadditive approach to integer
  programming.
\newblock {\em Annals of Discrete Mathematics}, 5:97--112, 1979.

\bibitem{johnson:1981}
E.~L. Johnson.
\newblock Characterization of facets for multiple right-hand side choice linear
  programs.
\newblock {\em Math.~Program.~Study}, 14:112--142, 1981.

\bibitem{johnson:ne:sa:2000}
E.~L. Johnson, G.~L. Nemhauser, and M.~W.~P. Savelsbergh.
\newblock Progress in linear programming-based algorithms for integer
  programming: an exposition.
\newblock {\em INFORMS J.~Comput.}, 12:2--23, 2000.

\bibitem{johnson:pa:1982}
E.~L. Johnson and M.~W. Padberg.
\newblock Degree-two inequalities, clique facets and biperfect graphs.
\newblock {\em Ann.~Disc.~Math.}, 16:169--187, 1982.

\bibitem{learnBranching}
E.~B. Khalil, P.~L. Bodic, L.~Song, G.~L. Nemhauser, and B.~Dilkina.
\newblock Learning to branch in mixed integer programming.
\newblock In {\em Proceedings of the Thirtieth AAAI Conference on Artificial
  Intelligence}, AAAI'16, pages 724--731. AAAI Press, 2016.

\bibitem{MIPLIB2010}
T.~Koch, T.~Achterberg, E.~Andersen, O.~Bastert, T.~Berthold, R.~E. Bixby,
  E.~Danna, G.~Gamrath, A.~M. Gleixner, S.~Heinz, A.~Lodi, H.~D. Mittelmann,
  T.~K. Ralphs, D.~Salvagnin, D.~E. Steffy, and K.~Wolter.
\newblock {MIPLIB} 2010.
\newblock {\em Math.~Program.~Comput.}, 3(2):103--163, 2011.

\bibitem{koppe2015electronic}
M.~K{\"o}ppe and Y.~Zhou.
\newblock An electronic compendium of extreme functions for the gomory--johnson
  infinite group problem.
\newblock {\em Oper.~Res.~Lett.}, 43(4):438--444, 2015.

\bibitem{koppe2017new}
M.~K{\"o}ppe and Y.~Zhou.
\newblock New computer-based search strategies for extreme functions of the
  gomory--johnson infinite group problem.
\newblock {\em Math.~Program.~Comput.}, 9(3):419--469, 2017.

\bibitem{koppe2017notions}
M.~K{\"o}ppe and Y.~Zhou.
\newblock On the notions of facets, weak facets, and extreme functions of the
  gomory--johnson infinite group problem.
\newblock In {\em International Conference on Integer Programming and
  Combinatorial Optimization}, pages 330--342. Springer, 2017.

\bibitem{lubbeckeLearning}
M.~Kruber, M.~E. L{\"u}bbecke, and A.~Parmentier.
\newblock Learning when to use a decomposition.
\newblock In D.~Salvagnin and M.~Lombardi, editors, {\em Integration of AI and
  OR Techniques in Constraint Programming}, pages 202--210, Cham, 2017.
  Springer International Publishing.

\bibitem{kurpisz2015hardest}
A.~Kurpisz, S.~Lepp{\"{a}}nen, and M.~Mastrolilli.
\newblock On the hardest problem formulations for the 0/1 lasserre hierarchy.
\newblock volume~42, pages 135--143, 2017.

\bibitem{lasserre2001global}
J.~B. Lasserre.
\newblock Global optimization with polynomials and the problem of moments.
\newblock {\em SIAM J.~Optim.}, 11(3):796--817, 2001.

\bibitem{laurent2003comparison}
M.~Laurent.
\newblock A comparison of the {Sherali-Adams}, {Lov{\'a}sz-Schrijver}, and
  {Lasserre} relaxations for 0--1 programming.
\newblock {\em Math.~Oper.~Res.}, 28(3):470--496, 2003.

\bibitem{lebair}
T.~M. Lebair and A.~Basu.
\newblock Approximation of minimal functions by extreme functions.
\newblock {\em arXiv preprint arXiv:1708.04344}, 2017.

\bibitem{letchford:lo:2002b}
A.~N. Letchford and A.~Lodi.
\newblock Strengthening {Chv\'atal-Gomory} cuts and {G}omory fractional cuts.
\newblock {\em Oper.~Res.~Lett.}, 30(2):74--82, 2002.

\bibitem{li:ri:2008}
Y.~Li and J.-P.~P. Richard.
\newblock Cook, {K}annan and {S}chrijve{r's} example revisited.
\newblock {\em Discrete Optim.}, 5:724--734, 2008.

\bibitem{lodiCGCounting}
A.~Lodi, G.~Pesant, and L.-M. Rousseau.
\newblock On counting lattice points and chv{\'a}tal-gomory cutting planes.
\newblock In T.~Achterberg and J.~C. Beck, editors, {\em Integration of AI and
  OR Techniques in Constraint Programming for Combinatorial Optimization
  Problems}, pages 131--136, Berlin, Heidelberg, 2011. Springer Berlin
  Heidelberg.

\bibitem{Lodi2017}
A.~Lodi and G.~Zarpellon.
\newblock On learning and branching: a survey.
\newblock {\em TOP}, 25(2):207--236, Jul 2017.

\bibitem{louveaux2014algorithm}
Q.~Louveaux and L.~Poirrier.
\newblock An algorithm for the separation of two-row cuts.
\newblock {\em Math.~Program.}, 143(1-2):111--146, 2014.

\bibitem{lovasz:1989}
L.~Lov\'asz.
\newblock Geometry of numbers and integer programming.
\newblock {\em Mathematical Programming: Recent Developments and Applications},
  pages 177--210, 1989.

\bibitem{Lovasz91}
L.~Lov\'{a}sz and A.~Schirjver.
\newblock {Cones of matrices and set-functions and 0-1 optimization}.
\newblock {\em SIAM J.~Optim.}, {1}:{166--190}, {1991}.

\bibitem{marchand:ma:we:wo:2002}
H.~Marchand, A.~Martin, R.~Weismantel, and L.~A. Wolsey.
\newblock Cutting planes in integer and mixed integer programming.
\newblock {\em Discr.~Appl.~Math.}, 123:397--446, 2002.

\bibitem{marchand:wo:1999}
H.~Marchand and L.~A. Wolsey.
\newblock The 0-1 knapsack problem with a single continuous variable.
\newblock {\em Math.~Program.}, 85:15--33, 1999.

\bibitem{margotSafe}
F.~Margot.
\newblock Testing cut generators for mixed-integer linear programming.
\newblock {\em Math.~Program.~Comput.}, 1(1):69--95, Jul 2009.

\bibitem{Mastrolilli17}
M.~Mastrolilli.
\newblock High degree sum of squares proofs, bienstock-zuckerberg hierarchy and
  {CG} cuts.
\newblock In Friedrich Eisenbrand and Jochen K{\"{o}}nemann, editors, {\em
  Integer Programming and Combinatorial Optimization - 19th International
  Conference, {IPCO} 2017, Waterloo, ON, Canada, June 26-28, 2017,
  Proceedings}, volume 10328 of {\em Lecture Notes in Computer Science}, pages
  405--416. Springer, 2017.

\bibitem{miller2008new}
L.~A Miller, Y.~Li, and J.-P.~P Richard.
\newblock New inequalities for finite and infinite group problems from
  approximate lifting.
\newblock {\em Naval~Res.~Logist.}, 55(2):172--191, 2008.

\bibitem{molinaro2013understanding}
Marco Molinaro.
\newblock {\em Understanding the Strength of General-Purpose Cutting Planes}.
\newblock PhD thesis, Carnegie Mellon University, 2013.

\bibitem{moran2012strong}
Diego~A Mor{\'a}n~R, Santanu~S Dey, and Juan~Pablo Vielma.
\newblock A strong dual for conic mixed-integer programs.
\newblock {\em SIAM Journal on Optimization}, 22(3):1136--1150, 2012.

\bibitem{leeLearning}
G.~Nannicini, P.~Belotti, J.~Lee, J.~Linderoth, F.~Margot, and A.~W{\"a}chter.
\newblock A probing algorithm for minlp with failure prediction by svm.
\newblock In T.~Achterberg and J.~C. Beck, editors, {\em Integration of AI and
  OR Techniques in Constraint Programming for Combinatorial Optimization
  Problems}, pages 154--169, Berlin, Heidelberg, 2011. Springer Berlin
  Heidelberg.

\bibitem{NemWolBook}
G.~L. Nemhauser and L.~A. Wolsey.
\newblock {\em Integer and combinatorial optimization}.
\newblock Wiley-Interscience, 1988.

\bibitem{nemhauser:wo:1990}
G.~L. Nemhauser and L.~A. Wolsey.
\newblock A recursive procedure to generate all cuts for 0-1 mixed integer
  programs.
\newblock {\em Math.~Program.}, 46:379--390, 1990.

\bibitem{neto}
José Neto.
\newblock A simple finite cutting plane algorithm for integer programs.
\newblock {\em Operations Research Letters}, 40(6):578 -- 580, 2012.

\bibitem{Owen2001}
J.~H. Owen and S.~Mehrotra.
\newblock A disjunctive cutting plane procedure for general mixed-integer
  linear programs.
\newblock {\em Math.~Program.}, 89(3):437--448, Feb 2001.

\bibitem{padberg:va:wo:1985}
M.~W. Padberg, T.~J.~Van Roy, and L.~A. Wolsey.
\newblock Valid linear inequalities for fixed charge problems.
\newblock {\em Oper.~Res.}, 33:842--861, 1985.

\bibitem{Gill:mu:sa:wr}
M.A. Saunders M.H.~Wright P.E.~Gill, W.~Murray.
\newblock Sparse matrix methods in optimization.
\newblock {\em SIAM Journal on Scientific and Statistical Computing},
  5:562--589, 1984.

\bibitem{PiaWW11}
A.~Del Pia, C.~Wagner, and R.~Weismantel.
\newblock A probabilistic comparison of the strength of split, triangle, and
  quadrilateral cuts.
\newblock {\em Oper.~Res.~Lett.}, 39(4):234--240, 2011.

\bibitem{pokutta2010rank}
S.~Pokutta and A.~S Schulz.
\newblock On the rank of cutting-plane proof systems.
\newblock In {\em IPCO}, pages 450--463. Springer, 2010.

\bibitem{PokuttaS11}
S.~Pokutta and G.~Stauffer.
\newblock Lower bounds for the chv{\'{a}}tal-gomory rank in the 0/1 cube.
\newblock {\em Oper.~Res.~Lett.}, 39(3):200--203, 2011.

\bibitem{pokutta2011200}
Sebastian Pokutta and Gautier Stauffer.
\newblock Lower bounds for the chvátal–gomory rank in the 0/1 cube.
\newblock {\em Operations Research Letters}, 39(3):200 -- 203, 2011.

\bibitem{pudlak1997}
Pavel Pudl\'ak.
\newblock Lower bounds for resolution and cutting plane proofs and monotone
  computations.
\newblock {\em J. Symbolic Logic}, 62(3):981--998, 09 1997.

\bibitem{RichardDey}
J.-P.~P. Richard and S.~S. Dey.
\newblock The group-theoretic approach in mixed integer programming.
\newblock chapter~19, pages 727--801. Springer, 2010.

\bibitem{richard:li:mi:2009}
J.-P.~P. Richard, Y.~Li, and L.~A. Miller.
\newblock Valid inequalities for {MIPs} and group polyhedra from approximate
  liftings.
\newblock {\em Math.~Program.}, 118:253--277, 2009.

\bibitem{rothvoss2013lasserre}
T.~Rothvo{\ss}.
\newblock The lasserre hierarchy in approximation algorithms.
\newblock {\em Lecture Notes for the MAPSP}, pages 1--25, 2013.

\bibitem{RothvossS17}
T.~Rothvo{\ss} and L.~Sanit{\`{a}}.
\newblock 0/1 polytopes with quadratic {C}hv{\'{a}}tal rank.
\newblock {\em Oper.~Res.}, 65(1):212--220, 2017.

\bibitem{sanjeevi2012mixed}
S.~Sanjeevi and K.~Kianfar.
\newblock Mixed n-step {MIR} inequalities: Facets for the n-mixing set.
\newblock {\em Discrete Optim.}, 9(4):216--235, 2012.

\bibitem{Schrijver80}
A.~Schrijver.
\newblock On cutting planes.
\newblock {\em Ann.~Disc.~Math.}, 9:291--296, 1980.
\newblock Combinatorics 79 (Proc. Colloq., Univ. Montr{\'e}al, Montreal, Que.,
  1979), Part II.

\bibitem{Sen2005}
S.~Sen and J.~Higle.
\newblock The c3 theorem and a d2 algorithm for large scale stochastic
  mixed-integer programming: Set convexification.
\newblock {\em Math.~Program.}, 104(1):1--20, 2005.

\bibitem{sherali:ad:1990}
H.~D. Sherali and W.~P. Adams.
\newblock A hierarchy of relaxations between the continuous and convex hull
  representations for zero-one programming problems.
\newblock {\em SIAM J.~Optim.}, 3:411--430, 1990.

\bibitem{sherali:ad:1994}
H.~D. Sherali and W.~P. Adams.
\newblock A hierarchy of relaxations and convex hull characterizations for
  mixed-integer zero-one programming problems.
\newblock {\em Discr.~Appl.~Math.}, 52:83--106, 1994.

\bibitem{sherali:sm:ad:2000}
H.~D. Sherali, J.C. Smith, and W.~P. Adams.
\newblock Reduced first-level representations via the
  reformulation-linearization technique: results, counterexamples, and
  computations.
\newblock {\em Discr.~Appl.~Math.}, 101:247--267, 2000.

\bibitem{singh2010improving}
M.~Singh and K.~Talwar.
\newblock Improving integrality gaps via chv{\'a}tal-gomory rounding.
\newblock In {\em APPROX-RANDOM}, pages 366--379. Springer, 2010.

\bibitem{van1986valid}
T.~J. Van~Roy and L.~A Wolsey.
\newblock Valid inequalities for mixed 0--1 programs.
\newblock {\em Discr.~Appl.~Math.}, 14(2):199--213, 1986.

\bibitem{Walter14}
M.~Walter.
\newblock Sparsity of lift-and-project cutting planes.
\newblock In {\em Operations Research Proceedings 2012}, pages 9--14. Springer,
  2014.

\bibitem{weismantel19970}
R.~Weismantel.
\newblock On the 0/1 knapsack polytope.
\newblock {\em Math.~Program.}, 77(3):49--68, 1997.

\bibitem{wesselmann}
F.~Wesselmann and U.~H. Suhl.
\newblock Implementing cutting plane management and selection techniques.
\newblock Technical report, University of Paderborn, December 2012.

\bibitem{wolsey:1975}
L.~A. Wolsey.
\newblock Faces for a linear inequality in 0-1 variables.
\newblock {\em Math.~Program.}, 8:165--178, 1975.

\bibitem{wolsey1989submodularity}
L.~A Wolsey.
\newblock Submodularity and valid inequalities in capacitated fixed charge
  networks.
\newblock {\em Oper.~Res.~Lett.}, 8(3):119--124, 1989.

\bibitem{wolseybook}
L.~A. Wolsey.
\newblock {\em Integer Programming}.
\newblock Wiley-interscience, New York, 1998.

\bibitem{MLijcai}
L.~Xu, F.~Hutter, H.H. Hoos, and K.~Leyton{-}Brown.
\newblock Hydra{-}mip{:} automated algorithm configuration and selection for
  mixed{-}integer programming.
\newblock In {\em {RCRA} Workshop on Experimental Evaluation of Algorithms for
  Solving Problems with Combinatorial Explosion at the International Joint
  Conference on Artificial Intelligence {(IJCAI)}}, 2011.

\bibitem{YildizC16}
S.~Yildiz and G.~Cornu{\'{e}}jols.
\newblock Cut-generating functions for integer variables.
\newblock {\em Math.~Oper.~Res.}, 41(4):1381--1403, 2016.

\bibitem{zanette2011lexicography}
A.~Zanette, M.~Fischetti, and E.~Balas.
\newblock Lexicography and degeneracy: can a pure cutting plane algorithm work?
\newblock {\em Math.~Program.}, 130(1):153--176, 2011.

\bibitem{zemel:1978}
E.~Zemel.
\newblock Lifting the facets of zero-one polytopes.
\newblock {\em Math.~Program.}, 15:268--277, 1978.

\bibitem{ZhangK14}
M.~Zhang and S.~K{\"{u}}{\c{c}}{\"{u}}kyavuz.
\newblock Finitely convergent decomposition algorithms for two-stage stochastic
  pure integer programs.
\newblock {\em SIAM J.~Optim.}, 24(4):1933--1951.

\end{thebibliography}

\end{document}